\documentclass[a4paper,11pt]{article}

\usepackage{mathrsfs} 
\usepackage{stmaryrd}

\usepackage{amsthm}
\usepackage{hyperref}
\usepackage{amssymb}
\usepackage{latexsym}
\usepackage{amsmath}
\usepackage{amsfonts}
\usepackage{mathabx}
\usepackage[dvips]{graphicx}
\usepackage[utf8]{inputenc}
\usepackage{xcolor}

\usepackage[T1]{fontenc}
\usepackage[french,english]{babel}
\usepackage{url}
\usepackage{enumerate}
\usepackage{hyperref}
\usepackage{color}
\usepackage{bbm}

\usepackage{fancyhdr}
\newtheorem{Theorem}{Theorem}[part]

\newtheorem{Proposition}{Proposition}[part]

\newtheorem{Assumption}{Assumption}[part]
\newtheorem{Lemma}{Lemma}[part]
\newtheorem{Corollary}{Corollary}[part]
\newtheorem{Remark}{Remark}[part]
\newtheorem{Example}{Example}[part]

\makeatletter \@addtoreset{equation}{section}

\@addtoreset{Definition}{section}

\@addtoreset{Theorem}{section}

\@addtoreset{Proposition}{section}

\@addtoreset{Property}{section}

\@addtoreset{Assumption}{section}

\@addtoreset{Corollary}{section}

\@addtoreset{Lemma}{section}

\@addtoreset{Remark}{section}

\@addtoreset{Example}{section}

\newcommand{\cA}{\mathcal{A}}
\newcommand{\cB}{\mathcal{B}}
\newcommand{\cC}{\mathcal{C}}
\newcommand{\cD}{\mathcal{D}}

\newcommand{\cF}{\mathcal{F}}
\newcommand{\cG}{\mathcal{G}}
\newcommand{\cH}{\mathcal{H}}
\newcommand{\cI}{\mathcal{I}}
\newcommand{\cJ}{\mathcal{J}}
\newcommand{\cK}{\mathcal{K}}

\newcommand{\cM}{\mathcal{M}}

\newcommand{\cP}{\mathcal{P}}

\newcommand{\cV}{\mathcal{V}}

\newcommand{\cY}{\mathcal{Y}}
\newcommand{\cZ}{\mathcal{Z}}

\newcommand{\F}{\mathbb{F}}
\newcommand{\G}{\mathbb{G}}
\renewcommand{\P}{\mathbb{P}}
\renewcommand{\H}{\mathbb{H}}

\newcommand{\R}{\mathbb{R}}

\newcommand{\N}{\mathbb{N}}
\newcommand{\A}{\mathbb{A}}

\newcommand{\fn}{\mathfrak n}

\def \proof{{\noindent \bf Proof. }}
\def \eproof{\hbox{ }\hfill$\Box$}

\newcommand{\ud}{\mathrm{d}}

\newcommand{\1}{{\bf 1}}
    
\newcommand{\set}[1]
    {\ensuremath{\{ #1 \}}}
    
\newcommand{\HP}[1] 
    {\ensuremath{\mathscr{H}^{#1}}}

\newcommand{\ind}[1]{1_{ \left\{ #1 \right\}}}

\newcommand{\esp}[1]{\ensuremath{\mathbb{E} \!\! \left[#1\right] }}

 \newcommand{\espcond}[2]{\ensuremath{\mathbb E\!\!\left[\left. #1 \right| #2\right]}}
 \newcommand{\probcond}[2]{\ensuremath{\mathbb P\!\!\left(#1 | #2\right)}}

\newcommand{\Dslice}{\cD_{\!\circ}}

\DeclareMathOperator*{\esssup}{ess\,sup}
\DeclareMathOperator*{\argmax}{arg\,max}

\DeclareMathOperator{\Det}{Det}
\DeclareMathOperator{\Trace}{Tr}
\title{Switching problems with controlled randomisation and associated obliquely reflected BSDEs}

\author{Cyril B\'en\'ezet\thanks{Centre de Mathématiques Appliquées (CMAP), Ecole Polytechnique and CNRS, Université Paris-Saclay, Route de Saclay, 91128 Palaiseau Cedex, France ({\tt cyril.benezet@polytechnique.edu})},
  Jean-Fran{\c{c}}ois Chassagneux\thanks{UFR de Math{\'e}matiques \& LPSM, Universit{\'e} de Paris, B{\^a}timent Sophie Germain, 8 place Aur{\'e}lie Nemours, 75013 Paris, France ({\tt chassagneux@lpsm.paris})}, Adrien Richou\thanks{UMR 5251 \& IMB, Universit{\'e} de Bordeaux, F-33400 Talence, France ({\tt adrien.richou@math.u-bordeaux.fr})}}

\begin{document}

\maketitle

\begin{abstract}
{
We introduce and study a new class of optimal switching problems, namely \emph{switching problem with controlled randomisation}, where some extra-randomness impacts the choice of switching modes and associated costs. We show that the optimal value of the switching problem is related to a new class of multidimensional obliquely reflected BSDEs. These BSDEs allow as well to construct an optimal strategy and thus to solve completely the initial problem. The other main contribution of our work is to prove new existence and uniqueness results for these obliquely reflected BSDEs. This is achieved by a careful study of the domain of reflection and the construction of an appropriate oblique reflection operator in order to invoke results from \cite{chassagneux2018obliquely}.
}
\end{abstract}


\section{Introduction}

In this work, we introduce and study a new class of optimal switching problems in stochastic control theory.
%
\noindent The interest in switching problems comes mainly from their connections to financial and economic problems, like the pricing of \emph{real options} \cite{carmona2010valuation}. In a celebrated article \cite{HJ07}, Hamadène and Jeanblanc study the fair valuation of  a company producing electricity. In their work, the company management can choose between  two modes of production for their power plant --operating or close-- and a time of switching from one state to another, in order to maximise its expected return. Typically, the company will buy electricity on the market if the power station is not operating. The company  receives a profit for delivering electricity in each regime. The main point here is that a fixed cost  penalizes the profit upon switching. This \emph{switching problem} has been generalized to more than two modes of production \textcolor{black}{\cite{DHP09}}. Let us now discuss this \emph{switching problem} with $d \ge 2$ modes in more details.
The costs to switch from one state to another are given by a matrix $(c_{i,j})_{1 \le i,j \le d}$. The management  optimises the expected company profits by choosing switching strategies which are sequences of stopping times $(\tau_n)_{n \ge 0}$ and modes $(\zeta_n)_{n \ge 0}$. The current state of the strategy is given by
$a_t = \sum_{k = 0}^{+\infty} \zeta_k 1_{[\tau_k, \tau_{k+1})}(t), \; t \in [0,T]\;,$
where $T$ is a terminal time. To formalise the problem, we assume that we are working on a complete probability space $(\Omega,\cA,\P)$ supporting a Brownian Motion $W$. The stopping times are defined with respect to the filtration $(\cF_t)_{t \ge 0}$ generated by this Brownian motion.
Denoting by $f(t,i)$ the instantaneous profit received at time $t$ in mode $i$, the time cumulated profit associated to a switching strategy is given by $\int_{0}^Tf(t,a_t) \ud t - \sum_{k = 0}^{+\infty} c_{\zeta_k, \zeta_{k+1}} \ind{\tau_{k+1} \le t\wedge T} $. The management solves then  at the initial time the following control problem
\begin{align}\label{eq de classical switching problem}
\mathcal{V}_0 = \sup_{a \in \mathscr{A}} \esp{\int_{0}^Tf(t,a_t) \ud t - \sum_{k = 0}^{+\infty} c_{\zeta_k, \zeta_{k+1}} \ind{\tau_{k+1} \le T}}\,,
\end{align}
where $\mathscr{A}$ is a set of admissible strategies that will be precisely described below (see Section \ref{problem}).
We shall refer to problems of the form \eqref{eq de classical switching problem} under the name of \emph{classical switching problems}. These problems have received a lot of interest and are now quite well understood \textcolor{black}{\cite{HJ07,DHP09,HT10,CEK11}}. In our work, we introduce a new kind of switching problem, to model more realistic situations, by taking into account uncertainties that are encountered in practice.
Coming back to the simple but enlightning example of an electricity producer described in \cite{HJ07}, we introduce some extra-randomness in the production process. Namely, when switching to the operating mode, it may happen with --hopefully-- a small probability that the station will have some dysfunction.
This can be represented by a new mode of ``production'' with a greater switching cost than the business as usual one. To capture this phenomenon in our mathematical model, we introduce a randomisation procedure: the management  decides the time of switching but  the mode is chosen randomly according to some extra noise source. 
We shall refer to this kind of problem by \emph{randomised switching problem}. However, we do not limit our study to this framework. Indeed, we allow some control by the agent on this randomisation. Namely, the agent can chose optimally a probability distributions $P^u$ on the modes space given some parameter $u \in \cC$, in the control space. The new mode $\zeta_{k+1}$ is then drawn, independently of everything up to now, according to this distribution $P^u$ and a specific switching cost  $c^u_{\zeta_k,\zeta_{k+1}}$ is applied. The management strategy is thus given now by the sequence $(\tau_k,u_k)_{k \ge 0}$ of switching times and controls. The maximisation problem is still given by \eqref{eq de classical switching problem}. \textcolor{black}{Let us observe however 
that $\esp{c^{u_k}_{\zeta_k,\zeta_{k+1}}}=\esp{\sum_{1\leqslant j\leqslant d} P^{u_k}_{\zeta_k,j}c^{u_k}_{\zeta_k,j}}$, thanks to the tower property of conditional expectation. In particular, we will only work with the mean switching costs $\bar{c}^{u_k}_i:=\sum_{1\leqslant j\leqslant d} P^{u_k}_{i,j}c^{u_k}_{i,j}$ in \eqref{eq de classical switching problem}.} 
We name this kind of control problem \emph{switching problem with controlled randomisation}. 
Although their apparent modeling power, this kind of control problem has not been considered in the literature before, to the best of our knowledge. In particular, we will show that the \emph{classical} or \emph{randomised switching problem} are just special instances of this more generic problem. The \emph{switching problem with controlled randomisation} is introduced rigorously in Section \ref{problem} below.

A key point in our work is to relate the control problem under consideration to a new class of obliquely reflected Backward Stochastic Differential Equations (BSDEs). In the first part, following the approach of \textcolor{black}{\cite{HJ07,DHP09,HT10}}, we completely solve the \emph{switching problem with controlled randomisation} by providing an optimal strategy. The optimal strategy is built by using the solution to a well chosen obliquely reflected BSDE. Although this approach is not new, the link between the obliquely reflected BSDE and the switching problem is more subtle than in the classical case due to the state uncertainty. In particular, some care must be taken when defining the adaptedness property of the strategy and associated quantities. Indeed, a tailor-made filtration,  studied in details in Appendix A.2, is associated to each admissible strategy. The state and cumulative cost processes are adapted to this filtration, and the associated reward process is defined as the $Y$-component of the solution to some ``switched'' BSDE in this filtration.
The classical estimates used to identify an optimal strategy have to be adapted to take into account the extra orthogonal martingale arising when solving this ``switched''  BSDE in a non Brownian filtration. 

\noindent In the second part of our work, we study \textcolor{black}{the auxiliary obliquely reflected BSDE, which is written in the Brownian filtration and represents the optimal value in all the possible starting modes.}
Reflected BSDEs were first considered by Gegout-Petit and Pardoux \cite{GPP96}, in a multidimensional setting of normal reflections. In one dimension, they have also been studied in \cite{EKPPQ97} in the  so called simply reflected case, and in \cite{CK96} in the doubly reflected case. The multidimensional RBSDE associated to the \emph{classical switching problem} is reflected in a specific convex domain and involves oblique directions of reflection. Due to the controlled randomisation, the domain in which the $Y$-component of the auxiliary RBSDE is constrained is different from the \emph{classical switching problem} domain and its shape varies a lot from one model specification to another. The existence of a solution to the obliquely reflected BSDE has thus to be studied carefully. We do so by relying on  the article \cite{chassagneux2018obliquely}, that studies, in a generic way, the obliquely reflected BSDE in a fixed convex domain in both Markovian and non-Markovian setting. The main step for us here is  to exhibit an oblique reflection operator, with the good 
properties  to use the results in \cite{chassagneux2018obliquely}. We are able to obtain new existence results for this class of obliquely reflected BSDEs. Because we are primarily interested in solving the control problem, we derive the uniqueness of the obliquely reflected BSDEs in the Hu and Tang specification for the driver \cite{HT10}, namely $f^{i}(t,y,z) := f^{i}(t,y^i,z^i)$ for $i \in \set{1,\dots,d}$.
But our results can be easily generalized to the specification $f^{i}(t,y,z) := f^{i}(t,y,z^i)$ by using similar arguments as in \cite{CEK10}. 

The rest of the paper is organised as follows. In Section \ref{game}, we introduce the  \emph{switching problem with controlled randomisation}.
We prove that, if there exists a solution to the associated BSDE with oblique reflections, then its  $Y$-component coincides with the value of the switching problem. A verification argument allows then to deduce uniqueness of the solution of the obliquely reflected BSDE. In Section \ref{existence}, we show that there exists indeed a solution to the obliquely reflected BSDE under some conditions on the parameters of the switching problem and its randomisation. We also prove uniqueness of the solution under some structural condition on the driver $f$.
Finally, we gather in the Appendix section some technical results.

\paragraph{Notations}
If $n \ge 1$, we let $\cB^n$ be the Borelian sigma-algebra on $\R^n$.\\
For any filtered probability space $(\Omega,\cG,\F,\P)$ and constants $T > 0$ and $p \ge 1$, we define the following spaces:
\begin{itemize}
\item $L^p_n(\cG)$ is the set of $\cG$-measurable random variables $X$ valued in $\R^n$ satisfying $\esp{|X|^p} < +\infty$,
\item $\cP(\F)$ is the predictable sigma-algebra on $\Omega \times [0,T]$,
\item $\H^p_n(\F)$ is the set of predictable processes $\phi$ valued in $\R^n$ such that
  \begin{align}
    \esp{\int_0^T |\phi_t|^p \ud t} < +\infty,
  \end{align}
\item $\mathbb S^p_n(\F)$ is the set of \textcolor{black}{c\`adl\`ag} processes $\phi$ valued in $\R^n$ such that
  \begin{align}
    \esp{\sup_{0 \le t \le T} |\phi_t|^p} < +\infty,
  \end{align}
\item $\A^p_n(\F)$ is the set of continuous processes $\phi$ valued in $\R^n$ such that $\phi_T \in L^p_n(\cF_T)$ and $\phi^i$ is nondecreasing for all $i = 1, \dots, n$.
\end{itemize}
If $n = 1$, we omit the subscript $n$ in previous notations.\\[.3cm]
For $d \ge 1$, we denote by $(e_i)_{i=1}^d$ the canonical basis of $\R^d$ and $S_d(\R)$ the set of symmetric matrices of size $d \times d$ with real coefficients.\\[.3cm]
If $\cD$ is a convex subset of $\R^d$ ($d \ge 1$) and $y \in \cD$, we denote by $\cC(y)$ the outward normal cone at $y$, defined by
\begin{align}
  \cC(y) := \{ v \in \R^d : v^\top (z - y) \le 0 \mbox{ for all } z \in \cD \}.
\end{align}
We also set $\fn(y) := \cC(y) \cap \{v \in \R^d : |v| = 1\}$.\\[.3cm]
If $X$ is a matrix of size $n \times m$, $\cI \subset \{1,\dots,n\}$ and $\cJ \subset \{1,\dots,m\}$, we set $X^{(\cI,\cJ)}$ the matrix of size $(n - |\cI|) \times (m - |\cJ|)$ obtained from $X$ by deleting rows with index $i \in \cI$ and columns with index $j \in \cJ$. If $\cI = \{i\}$ we set $X^{(i,\cJ)} := X^{(\cI,\cJ)}$, and similarly if $\cJ = \{j\}$.\\
If $v$ is a vector of size $n$ and $1 \le i \le n$, we set $v^{(i)}$ the vector of size $n-1$ obtained from $v$ by deleting coefficient $i$.

\newcommand{\ls}[2]{\ensuremath{#1^{(#2)} }}
\newcommand{\rs}[2]{\ensuremath{#1+ \1_{\set{#1\ge #2}} }}

\noindent For $(i,j) \in \set{1,\dots,d}$, we define
$ \ls{i}{j}:= i - \1_{\set{i>j}} \in \set{1,\dots,d-1}\;$, for $d\ge 2$.
\\
We denote by $\succcurlyeq$ the component by component partial ordering relation on vectors and matrices.

\section{Switching problem with controlled randomisation} \label{game}
\definecolor{green}{rgb}{0.0, 0.65, 0.31}
We introduce here a new kind of stochastic control problem that we name \emph{switching problem with controlled randomisation}. In contrast with the usual switching problems \cite{HJ07,HZ10,HT10}, the agent cannot choose directly the new state, but chooses a probability distribution under which the new state will be determined. In this section, we assume the existence of a solution to some auxiliary obliquely reflected BSDE to characterize the value process and an optimal strategy for the problem, see Assumption \ref{exist} below.

\vspace{2mm}
 Let $(\Omega, \cG, \P)$ be a probability space. We fix a finite time horizon $T > 0$ and $\kappa \ge 1,d \ge 2$ two integers.
We assume that there exists a $\kappa$-dimensional Brownian motion $W$ and a sequence $(\mathfrak{U}_n)_{n \ge 1}$ of independent random variables, independent of $W$, uniformly distributed on $[0,1]$. We also assume that $\cG$ is generated by the Brownian motion $W$ and the family $(\mathfrak{U}_n)_{n \ge 1}$.
We define $\F^0 = (\cF^0_t)_{t \ge 0}$ as the augmented Brownian filtration, which satisfies the usual conditions.\\
Let $\cC$ be an ordered compact metric space and $F : \cC \times \{1,\dots,d\} \times [0,1] \to \{1,\dots,d\}$ a measurable map.
To each $u \in \cC$ is associated a transition probability function on the state space $\{1, \dots, d\}$, given by $P^u_{i,j} := \P(F(u,i,\mathfrak{U}) = j)$ for $\mathfrak{U}$ uniformly distributed on $[0,1]$. {We assume that for all $(i,j) \in \{1,\dots,d\}^2$, the map $u \mapsto P^u_{i,j}$ is continuous.}\\
Let $\bar{c} : \{1,\dots,d\} \times \cC \to \R_+, (i, u) \mapsto \bar{c}_{i}^{u}$ a map such that $u \mapsto \bar{c}^u_{i}$ is continuous for all $i = 1, \dots, d$. 
We denote $\sup_{i \in \{1,\dots,d\}, u \in \cC} \bar{c}_{i}^{u} := \check{c}$ and $\inf_{i \in \{1,\dots,d\}, u \in \cC} \bar{c}_{i}^{u} := \hat{c}$.\\ 
Let $\xi = (\xi^1, \dots, \xi^d) \in L^2_d(\cF^0_T)$ and $f : \Omega \times [0,T] \times \R^d \times \R^{d \times \kappa} \to \R^d$ a map satisfying
\begin{itemize}
\item $f$ is $\cP(\F^0) \otimes \cB^{d} \otimes \cB^{d\times \kappa}$-measurable and $f(\cdot,0,0) \in \H^2_d(\F^0)$.
\item There exists $L \ge 0$ such that, for all $(t,y,y',z,z') \in [0,T] \times \R^d \times \R^d \times \R^{d \times \kappa} \times \R^{d \times \kappa}$, 
  \begin{align*} 
    |f(t,y,z)-f(t,y',z')| \le L(|y-y'|+|z-z'|).
  \end{align*}
\end{itemize}
These assumptions will be in force throughout our work. We shall also use, in this section only, the following additional assumptions.
{
\begin{Assumption}
 \label{hyp sup section 2}
 \begin{enumerate}[i)]
  \item  Switching costs are assumed to be positive, i.e. $\hat{c}>0$.
  \item For all $(t,y,z) \in [0,T]\times\R^d\times\R^{d\times \kappa}$, it holds almost surely,
  \begin{align} \label{hyp plus simple sur f}
    f(t,y,z) = (f^i(t,y^i,z^i))_{1\le i \le d}.
  \end{align}
 \end{enumerate}
\end{Assumption}
\begin{Remark} \label{re positive cost}
i) \textcolor{black}{ It is usual to assume  positive costs in the litterature on switching problem. In particular, the cumulative cost process, see \eqref{definition A et N}, is non decreasing.
Introducing signed costs adds extra technical difficulties in the proof of the representation theorem (see e.g. \cite{Martyr-16} and references therein). We postpone the adaptation of our results in this more general framework to future works.}
\\
ii) 
 Assumption \eqref{hyp plus simple sur f} is also classical since it allows to get a comparison result for BSDEs which is key to obtain the representation theorem. Note however than our results can be generalized to the case $f^i(t,y,z)=f^i(t,y,z^i)$ for $i \in \{1,...,d\}$ by using similar arguments as in \cite{CEK11}.
\end{Remark}
}

\subsection{Solving the control problem using obliquely reflected BSDEs} \label{problem}

We define in this section the stochastic optimal control problem. We first introduce the strategies available to the agent  and  related processes. The definition of the strategy is more involved than in the usual switching problem setting since its adaptiveness property is understood with respect to a filtration built recursively.

\vspace{2mm}
A strategy is thus given by $\phi = \left(\zeta_0, (\tau_n)_{n \ge 0}, (\alpha_n)_{n \ge 1}\right)$ where $\zeta_0 \in \{1,\dots,d\}$, $(\tau_n)_{n \ge 0}$ is a nondecreasing sequence of random times and $(\alpha_n)_{n \ge 1}$ is a sequence of $\cC$-valued random variables, which satisfy:
\begin{itemize}
\item $\tau_0 \in [0,T]$ and $\zeta_0 \in \{1, \dots, d\}$ are deterministic.
\item For all $n \ge 0$, $\tau_{n+1}$ is a $\F^n$-stopping time and $\alpha_{n+1}$ is $\cF^n_{\tau_{n+1}}$-measurable (recall that $\F^0$ is the augmented Brownian filtration). We then set $\F^{n+1} = (\cF^{n+1}_t)_{t \ge 0}$ with $\cF^{n+1}_t := \cF^n_t \vee \sigma(\mathfrak{U}_{n+1} \ind{\tau_{n+1} \le t})$.
\end{itemize}
Lastly, we define $\F^\infty = (\cF^\infty_t)_{t \ge 0}$ with $\cF^\infty_t := \bigvee_{n \ge 0} \cF^n_t, t \ge 0$.
\\
For a strategy $\phi = \left(\zeta_0, (\tau_n)_{n \ge 0}, (\alpha_n)_{n \ge 1}\right)$, we set, for  $n \ge 0$, 
\begin{align*}
 \zeta_{n+1} := F(\alpha_{n+1}, \zeta_n, \mathfrak{U}_{n+1}) \text{ and } a_t := \sum_{k = 0}^{+\infty} \zeta_k 1_{[\tau_k, \tau_{k+1})}(t), \; t \ge 0\;,
\end{align*}
which represents the state after a switch and the state process respectively.
We also introduce two processes, for $t \ge 0$,
\begin{align} \label{definition A et N}
A^\phi_t = \sum_{k = 0}^{+\infty} \bar{c}_{\zeta_k}^{\alpha_{k+1}} \ind{\tau_{k+1} \le t} \text{ and } N^\phi_t := \sum_{k \ge 0} \ind{\tau_{k+1} \le t}.
\end{align}
The random variable $A^\phi_t$ is the cumulative cost  up to time $t$
and  $N^\phi_t$ is the number of switches before time $t$. 
Notice that the processes $(a,A^\phi,N^\phi)$ are adapted to $\F^\infty$ \textcolor{black}{and that $A^\phi$ is a non decreasing process.}

\noindent We say that a strategy $\phi = (\zeta_0, (\tau_n)_{n \ge 0}, (\alpha_n)_{n \ge 1})$ is an \emph{admissible strategy} if the cumulative cost process satisfies 
\begin{align}\label{eq strategy admissible}
A^\phi_T - A^\phi_{\tau_0} \in L^2(\F^\infty_T) \quad\text{ and }\quad \espcond{\left(A^\phi_{\tau_0}\right)^2}{\cF^0_{\tau_0}} < +\infty \;a.s.
\end{align}
We denote by $\mathscr A$ the set of admissible strategies, and for $t \in [0,T]$ and $i \in \{1,\dots,d\}$, we denote by $\mathscr A^i_t$ the subset of admissible strategies satisfying $\zeta_0 = i$ and $\tau_0 = t$.

\vspace{2mm}
\begin{Remark}
\begin{enumerate}[i)]
 \item \textcolor{black}{The definition of an \emph{admissible strategy} is slightly weaker  than usual \cite{HT10}, which requires the stronger property $A^\phi_T \in L^2(\cF^\infty_T)$. But, importantly, the above definition is enough to define the \emph{switched} BSDE associated to an admissible control, see below. Moreover, we observe in the next section that optimal strategies are admissible with respect to our definition, but not necessarily with the usual one, due to possible simultaneous jumps at the initial time.}
\item  {For technical reasons involving possible simultaneous jumps, we cannot consider the generated filtration associated to $a$, which is contained in $\F^\infty$.
}
\end{enumerate}
\end{Remark}

\vspace{2mm}
\noindent We are now in position to introduce the reward associated to an admissible strategy. If $\phi = (\zeta_0,(\tau_n)_{n \ge 0},(\alpha_n)_{n\ge 1}) \in \mathscr A$, the reward is defined as the value $\espcond{U^\phi_{\tau_0}-A^\phi_{\tau_0}}{\cF^0_{\tau_0}}$, where $(U^\phi,V^\phi,M^\phi) \in \mathbb S^2(\F^\infty) \times \H^2_{\kappa}(\F^\infty) \times \H^2(\F^\infty)$ is the solution of the following ``switched'' BSDE \cite{HT10} on the filtered probability space $(\Omega,\cG,\F^\infty,\P)$:
\begin{align} \label{switched BSDE}
  U_t = \xi^{a_T} + \int_t^T f^{a_s}(s, U_s, V_s) \ud s - \int_t^T V_s \ud W_s - \int_t^T \ud M_s - \int_t^T \ud A^\phi_s, \quad t \in [\tau_0,T].
\end{align}
\begin{Remark}This switched BSDE rewrites as a classical BSDE in $\F^\infty$, and since $A^\phi_{\cdot} - A^\phi_t \in \mathbb S^2(\F^\infty)$, the terminal condition and the driver are standard parameters,  there exists a unique solution to \eqref{switched BSDE} for all $\phi \in \mathscr A$. We refer to Section \ref{sec BSDE} for more details.
\end{Remark}
\vspace{2mm}
\noindent For $t \in [0,T]$ and $i \in \{1,\dots,d\}$, the agent aims thus to solve the following maximisation problem:
\begin{align}
\mathcal V^i_t = \esssup_{\phi \in \mathscr A^i_t} \espcond{U^\phi_t - A^\phi_t}{\cF^0_t}.
\end{align}
\textcolor{black}{We first remark that this control problem corresponds to \eqref{eq de classical switching problem} as soon as $f$ does not depend on $y$ and $z$. Moreover, the term $ \espcond{A^\phi_t}{\cF^0_t}$ is non zero if and only if we have at least one instantaneous switch at initial time $t$. 
}

\noindent The main result of this section is the next theorem that relates the value process $\cV$ to the solution of an obliquely reflected BSDEs, that is introduced in the following assumption:
\begin{Assumption} \label{exist}
  \begin{enumerate}[i)]
  \item There exists a solution $(Y,Z,K) \in \mathbb S^2_d(\F^0) \times \H^2_{d \times \kappa}(\F^0) \times \A^2_d(\F^0)$ to the following obliquely reflected BSDE:
  \begin{align}
    \label{orbsde} Y^i_t &= \xi + \int_t^T f^i(s,Y^i_s,Z^i_s) \ud s - \int_t^T Z^i_s \ud W_s + \int_t^T \ud K^i_s, \quad t \in [0,T], \, i \in \cI,\\
    \label{orbsde2} Y_t &\in \cD, \quad t \in [0,T],\\
    \label{orbsde3} \int_0^T &\left(Y^i_t - \sup_{u \in \cC} \left\{ \sum_{j=1}^d P^u_{i,j} Y^j_t - \bar{c}_{i}^{u}\right\}\right) \ud K^i_t = 0, \quad i \in \cI,
  \end{align}
  where $\cI := \{1,\dots,d\}$ and $\cD$ is the following convex subset of $\R^d$:
  \begin{align}
    \cD := \left\{ y \in \R^d : y_i \ge \sup_{u \in \cC} \left\{ \sum_{j=1}^d P^u_{i,j} y_j - \bar{c}_{i}^{u}\right\}, i \in \cI \right\}. \label{domain}
  \end{align}
\item For all $u \in \cC$ and $i \in \{1,\dots,d\}$, we have $P^u_{i,i} \neq 1$.
\end{enumerate}
\end{Assumption}
\textcolor{black}{Let us observe that the positive costs assumption implies that $\cD$ has a non-empty interior. Except for Section \ref{section unicite}, this is the main setting for this part, recall Remark \ref{re positive cost}. In Section \ref{existence}, the system \eqref{orbsde}-\eqref{orbsde2}-\eqref{orbsde3} is studied in details in a general costs setting. An important step is then to understand when $\cD$ has non-empty interior.}
\vspace{2mm}

\begin{Theorem}\label{th representation result}
Assume that Assumptions \ref{hyp sup section 2} and \ref{exist} are satisfied.
  \begin{enumerate}
  \item For all $i \in \{1,\dots,d\}$, $t \in [0,T]$ and $\phi \in \mathscr A^i_t$, we have $Y^i_t \ge \espcond{U^\phi_t - A^\phi_t}{\cF^0_t}$.
  \item We have $Y^i_t = \espcond{U^{\phi^\star}_t - A^{\phi^\star}_t}{\cF^0_t}$, where $\phi^\star = (i,(\tau^\star_n)_{n \ge 0}, (\alpha^\star_n)_{n \ge 1})\in \mathscr A^i_t$ is defined in \eqref{def tau star}-\eqref{def alpha star}.
  \end{enumerate}
\end{Theorem}

\vspace{2mm}
\noindent The proof is given at the end of \textcolor{black}{Section \ref{sous section preuve th representation}}. It will use several lemmata that we introduce below.
We first remark, that as an immediate consequence, we obtain the uniqueness of the BSDE used to characterize the value process of the control problem.

\begin{Corollary} \label{co uniqueness RBSDE}
  Under Assumptions \ref{hyp sup section 2} and \ref{exist}, there exists a unique solution $(Y,Z,K) \in \mathbb S^2_d(\F^0) \times \H^2_{d\times \kappa}(\F^0) \times \A^2_d(\F^0)$ to the obliquely reflected BSDE \eqref{orbsde}-\eqref{orbsde2}-\eqref{orbsde3}.
\end{Corollary}
 
 
\textcolor{black}{\begin{Remark}
 The \emph{classical switching problem} is an example of \emph{switching problem with controlled randomisation}. Indeed, we just have to consider $\cC=\{1,...,d-1\}$, 
 $$P_{i,j}^u= \begin{cases}
               1 & \text{ if } j-i=u \text{ mod } d,\\
               0 & \text{ otherwise.}
              \end{cases},\quad \forall u \in \cC, \, 1 \leqslant i,j \leqslant d$$
and
$$c_{i,j}= \begin{cases} \bar{c}_i^{j-i} & \text{ if } j>i\\
                          \bar{c}_i^{j-i+d} & \text{ if } j<i\\
                          0 & \text{ if } j=i.
                          \end{cases}\quad \forall u \in \cC, \, 1 \leqslant i,j \leqslant d.$$
We observe that, in this specific case,  there is no extra-randomness introduced at each switching time and so there is no need to consider an enlarged filtration.  In this setting, Theorem \ref{th representation result} is already known and Assumption \ref{exist} is fulfilled, see e.g. \cite{HZ10,HT10}.
\end{Remark}
}

\subsection{Uniqueness of solutions to reflected BSDEs with general costs}
\label{section unicite}
In this section, we extend the uniqueness result of Corollary \ref{co uniqueness RBSDE}.
Namely, we consider the case where $\inf_{1\le i \le d, u \in \cC} \bar{c}_i^u = \hat{c}$ can be nonpositive, meaning that only Assumptions \ref{hyp sup section 2}-ii) and \ref{exist} hold here. Assuming that $\cD$ has a non empty interior, we are then able to show uniqueness to \eqref{orbsde}-\eqref{orbsde2}-\eqref{orbsde3} in Proposition \ref{prop unicite couts generaux} below.

\vspace{2mm}
\noindent Fix $y^0$ in the interior of $\cD$. It is clear that for all $1\le i \le d$,
\begin{align*}
  y^0_i > \sup_{u \in \cC} \left\{ \sum_{j=1}^d P^u_{i,j} y^0_j - \bar{c}_{i}^u \right\}.
\end{align*}
We set, for all $1\le i \le d$ and $u \in \cC$,
\begin{align*}
  \tilde c_{i}^u := y^0_i - \sum_{j=1}^d P^u_{i,j} y^0_j + \bar{c}_{i}^u > 0,
\end{align*}
so that $\hat{\tilde c} := \inf_{1\le i \le d, u \in \cC} \tilde{c}_i^u > 0$ by compactness of $\cC$.
We also consider the following set
\begin{align*}
  \tilde \cD := \left\{ \tilde y \in \R^d : \tilde y_i \ge \sup_{u \in \cC} \left\{ \sum_{j=1}^d P^u_{i,j} \tilde y_j - \tilde c_{i}^u \right\}, 1 \le i \le d \right\}.
\end{align*}
\begin{Lemma} Assume that $\cD$ has a non empty interior. Then, 
  \begin{align*}
    \tilde \cD = \left\{ y - y^0 : y \in \cD \right\}.
  \end{align*}
\end{Lemma}
\begin{proof}
  If $y \in \cD$, let $\tilde y := y - y^0$. For $1 \le i \le d$ and $u \in \cC$, we have
  \begin{align*}
    \tilde y_i = y_i - y^0_i &\ge \sum_{j=1}^d P^u_{i,j} y_j - \bar{c}_{i}^u - y^0_i 
                             = \sum_{j=1}^d P^u_{i,j} (y_j - y^0_j) - (\bar{c}_{i}^u + y^0_i - \sum_{j=1}^d P^u_{i,j} y^0_j) \\
                             &= \sum_{j=1}^d P^u_{i,j} \tilde y_j - \tilde c_{i}^u,
  \end{align*}
  hence $\tilde y \in \tilde \cD$. \\
  Conversely, let $\tilde y \in \tilde \cD$ and let $y := \tilde y + y^0$. We can show by the same kind of calculation that 
  $y \in \cD$.
  \eproof
\end{proof}
\begin{Proposition}
\label{prop unicite couts generaux}
  Assume that $\cD$ has a non empty interior. Under Assumptions \ref{hyp sup section 2}-ii) and \ref{exist}-ii), there exists at most one solution to \eqref{orbsde}-\eqref{orbsde2}-\eqref{orbsde3} in  $\mathbb S^2_d(\F^0) \times \H^2_{d\times \kappa}(\F^0) \times \A^2_d(\F^0)$.
\end{Proposition}
\begin{proof}
  Let us assume that $(Y^1,Z^1,K^1)$ and $(Y^2,Z^2,K^2)$ are two solutions to \eqref{orbsde}-\eqref{orbsde2}-\eqref{orbsde3}. We set $\tilde Y^1 := Y^1 - y_0$ and $\tilde Y^2 := Y^2 - y_0$. Then one checks easily that $(\tilde Y^1, Z^1, K^1)$ and $(\tilde Y^2, Z^2, K^2)$ are solutions to \eqref{orbsde}-\eqref{orbsde2}-\eqref{orbsde3} with terminal condition $\tilde \xi = \xi - y_0$, driver $\tilde f$ given by
  \begin{align*}
    \tilde f^i(t,\tilde y_i,z_i) := f^i(t,\tilde y_i + y^0_i, z_i),\quad 1 \le i \le d,\, t \in [0,T],\, \tilde y \in \R^d,\, z \in \R^{d \times \kappa},
  \end{align*}
  and domain $\tilde \cD$. This domain is associated to a randomised switching problem with $\hat{\tilde c} > 0$, hence Corollary \ref{co uniqueness RBSDE} gives that $(\tilde Y^1, Z^1, K^1) = (\tilde Y^2, Z^2, K^2)$ which  implies the uniqueness.
  \eproof
\end{proof}

\textcolor{black}{\subsection{Proof of the representation result}
\label{sous section preuve th representation}}

\noindent We prove here our main result for this part, namely Theorem \ref{th representation result}. It is divided in several steps.

\subsubsection{Preliminary estimates}

We first introduce auxiliary processes associated to an admissible strategy and prove some key integrability properties.

\vspace{2mm}

Let $i \in \{1,\dots,d\}$ and $t \in [0,T]$. We set, for $\phi \in \mathscr A^i_t$ and $t \le s \le T$,
\begin{align}
  \cY^\phi_s &:= \sum_{k \ge 0} Y^{\zeta_k}_s 1_{[\tau_k, \tau_{k+1})}(s), \\
  \cZ^\phi_s &:= \sum_{k \ge 0} Z^{\zeta_k}_s 1_{[\tau_k, \tau_{k+1})}(s), \\
  \cK^\phi_s &:= \sum_{k \ge 0} K^{\zeta_k}_s 1_{[\tau_k, \tau_{k+1})}(s), \\
  \cM^\phi_s &:= \sum_{k \ge 0} \left(Y^{\zeta_{k+1}}_{\tau_{k+1}} - \espcond{Y^{\zeta_{k+1}}_{\tau_{k+1}}}{\cF^k_{\tau_{k+1}}} \right) \ind{t < \tau_{k+1} \le s}, \\
  \label{def A ronde} \cA^\phi_s &= \sum_{k \ge 0} \left( Y^{\zeta_k}_{\tau_{k+1}} - \espcond{Y^{\zeta_{k+1}}_{\tau_{k+1}}}{\cF^k_{\tau_{k+1}}} + \bar{c}_{\zeta_k}^{\alpha_{k+1}} \right) \ind{t < \tau_{k+1} \le s}.
\end{align}
\begin{Remark}For all $k \ge 0$, since $\alpha_{k+1} \in \cF^k_{\tau_{k+1}}$, we have
  \begin{align}
    \espcond{Y^{\zeta_{k+1}}_{\tau_{k+1}}}{\cF^k_{\tau_{k+1}}} = \sum_{j=1}^d \probcond{\zeta_{k+1} = j}{\cF^k_{\tau_{k+1}}} Y^j_{\tau_{k+1}} = \sum_{j=1}^d P^{\alpha_k}_{\zeta_k, j} Y^j_{\tau_{k+1}}. \label{calcul esp cond}
  \end{align}
\end{Remark}

\begin{Lemma}
\label{lem estimes startegies gene}
Assume that assumption \eqref{exist} is satisfied.
For any admissible strategy $\phi \in \mathscr A^i_t$, $\cM^\phi$ is a square integrable martingale with $\cM^\phi_t = 0$. Moreover, $\cA^\phi$ is increasing and satisfies $\cA^\phi_T \in L^2(\cF^\infty_T)$. In addition,
    \begin{align} \label{mart en t}
      \espcond{\left(\sum_{k \ge 0} \left(Y^{\zeta_{k+1}}_{\tau_{k+1}} - \espcond{Y^{\zeta_{k+1}}_{\tau_{k+1}}}{\cF^k_{\tau_{k+1}}}\right)\ind{\tau_{k+1} \le t}\right)^2}{\cF^0_t} < +\infty \quad \text{a.s.}
    \end{align}
\end{Lemma}

\proof
 Let $\phi \in \mathscr A^i_t$. Using \eqref{def A ronde} and \eqref{calcul esp cond}, we have, for all $s \in [t,T]$,
    \begin{align}
      \cA^\phi_s = \sum_{k \ge 0} \left( Y^{\zeta_k}_{\tau_{k+1}} - \sum_{j=1}^d P^{\alpha_{k+1}}_{\zeta_k,j} Y^j_{\tau_{k+1}} + \bar{c}_{\zeta_k}^{\alpha_{k+1}} \right)\ind{t < \tau_{k+1} \le s},
    \end{align}
    which is increasing since each summand is positive as $Y \in \cD$.\\
    We have, for $t \le s \le T$,
    \begin{align} \label{debut}
      \cY^\phi_s - \cY^\phi_t &= \sum_{k \ge 0} \left( Y^{\zeta_k}_{\tau_{k+1} \wedge s} - Y^{\zeta_k}_{\tau_k \wedge s} \right) + \sum_{k \ge 0} \left( Y^{\zeta_{k+1}}_{\tau_{k+1}} - Y^{\zeta_k}_{\tau_{k+1}} \right) \ind{t < \tau_{k+1} \le s}.
    \end{align}
    Using \eqref{orbsde}, we get, for all $k \ge 0$,
    \begin{align*} 
      &\hspace{-1.9cm}Y^{\zeta_k}_{\tau_{k+1} \wedge s} - Y^{\zeta_k}_{\tau_k \wedge s} \\
      &\hspace{-1.9cm}= -\int_{\tau_k \wedge s}^{\tau_{k+1} \wedge s} f^{\zeta_k}(u,Y^{\zeta_k}_u,Z^{\zeta_k}_u) \ud u + \int_{\tau_k \wedge s}^{\tau_{k+1} \wedge s} Z^{\zeta_k}_u \ud W_u - \int_{\tau_k \wedge s}^{\tau_{k+1} \wedge s} \ud K^{\zeta_k}_u,
    \end{align*}
    Recalling $\zeta_k$ is $\mathcal{F}_{\tau_k}$-measurable. We also have, using \eqref{calcul esp cond}, for all $k \ge 0$,
    \begin{align*} 
      &Y^{\zeta_{k+1}}_{\tau_{k+1}} - Y^{\zeta_k}_{\tau_{k+1}} \\ 
                                                               &= \left( Y^{\zeta_{k+1}}_{\tau_{k+1}} - \espcond{Y^{\zeta_{k+1}}_{\tau_{k+1}}}{\cF^k_{\tau_{k+1}}}\right) - \left( Y^{\zeta_k}_{\tau_{k+1}} - \sum_{j=1}^d P^{\alpha_{k+1}}_{\zeta_k, j} Y^j_{\tau_{k+1}} + \bar{c}_{\zeta_k}^{\alpha_{k+1}} \right) + \bar{c}_{\zeta_k}^{\alpha_{k+1}}.
    \end{align*}
    Plugging the two previous equalities into \eqref{debut}, we get:
    \begin{align}
      \nonumber \cY^\phi_s - \cY^\phi_t = &\sum_{k \ge 0} \left( -\int_{\tau_k \wedge s}^{\tau_{k+1} \wedge s} f^{\zeta_k}(u,Y^{\zeta_k}_u,Z^{\zeta_k}_u) \ud u + \int_{\tau_k \wedge s}^{\tau_{k+1} \wedge s} Z^{\zeta_k}_u \ud W_u - \int_{\tau_k \wedge s}^{\tau_{k+1} \wedge s} \ud K^{\zeta_k}_u \right)\\
      \nonumber &+ \sum_{k \ge 0} \left( Y^{\zeta_{k+1}}_{\tau_{k+1}} - \espcond{Y^{\zeta_{k+1}}_{\tau_{k+1}}}{\cF^k_{\tau_{k+1}}}\right) \ind{t < \tau_{k+1} \le s}  + A^\phi_s - A^\phi_t\\
                                          &- \sum_{k \ge 0} \left( Y^{\zeta_k}_{\tau_{k+1}} - \sum_{j=1}^d P^{\alpha_{k+1}}_{\zeta_k, j} Y^j_{\tau_{k+1}} + \bar{c}_{\zeta_k}^{\alpha_{k+1}} \right) \ind{t < \tau_{k+1} \le s}.
    \end{align}
    By definition of $\cY^\phi, \cZ^\phi, \cK^\phi, \cM^\phi, \cA^\phi$, we obtain, for all $s \in [t,T]$,
    \begin{align}
      \nonumber \cY^\phi_s = &\xi^{a_T} + \int_s^T f^{a_u}(u,\cY^\phi_u,\cZ^\phi_u) \ud u - \int_s^T \cZ^\phi_u \ud W_u - \int_s^T \ud \cM^\phi_u- \int_s^T \ud A^\phi_u \\
      \label{eq apres t} &+ \left[ \left(\cA^\phi_T + \cK^\phi_T \right) - \left(\cA^\phi_s + \cK^\phi_s \right)\right].
    \end{align}
    For any $n \ge 1$, we consider the admissible strategy $\phi_n = (\zeta_0, (\tau^n_k)_{k \ge 0}, (\alpha^n_k)_{k \ge 1})$ defined by $\zeta^n_0 = i = \zeta_0, \tau^n_k = \tau_k, \alpha^n_k = \alpha_k$ for $k \le n$, and $\tau^n_k = T+1$ for all $k > n$. We set $\cY^n := \cY^{\phi_n}, \cZ^nv:= \cZ^{\phi_n}$, and so on.\\
    By \eqref{eq apres t} applied to the strategy $\phi^n$, we get, recalling that $\cA^n_t = 0$,
    \begin{align}
      \nonumber \cA^n_{\tau_n \wedge T} = &\cY^n_t - \cY^n_{\tau_n \wedge T} - \int_t^{\tau_n \wedge T} f^{a^n_s}(s,\cY^n_s,\cZ^n_s) \ud s + \int_t^{\tau_n \wedge T} \cZ^n_s \ud W_s\\ &+ \int_t^{\tau_n \wedge T} \ud \cM^n_s + \int_t^{\tau_n \wedge T} \ud A^n_s - \int_t^{\tau_n \wedge T} \ud \cK^n_s.
    \end{align}
    We obtain, for a constant $\Lambda > 0$,
    \begin{align}
      \esp{|\cA^n_{\tau_n \wedge T}|^2} \le& \Lambda\left(\esp{|\cY^n_t|^2 + |\cY^n_{\tau_n \wedge T}|^2 + \int_t^{\tau_n \wedge T} |f^{a^n_s}(s,\cY^n_s,\cZ^n_s)|^2 \ud s \right. \right. \\
      \nonumber &\left.\left. + \int_t^{\tau_n \wedge T}|\cZ^n_s|^2\ud s + \int_t^{\tau_n \wedge T} \ud [\cM^n]_s + (A^\phi_T - A^\phi_t)^2 + (\cK^n_T)^2}\right).
    \end{align}
    We have
    \begin{align*}
      \esp{|\cY^n_r|^2} &\le \sum_{j=1}^d \esp{|Y^j_r|^2} = \esp{|Y_r|^2} \le \esp{\sup_{t \le r \le T} |Y_r|^2} = \Vert Y \Vert_{\mathbb S^2_d(\F^0)}^2,\\
      \esp{\int_t^{\tau_n \wedge T}|\cZ^n_s|^2\ud s} &\leqslant \Vert Z\Vert_{\H^2_{d\times \kappa}(\F^0)},\\
      \esp{\int_t^{\tau_n \wedge T} |f^{a_s^n}(s,\cY^n_s,\cZ^n_s)|^2\ud s} 
                                                                   & \le 4L^2T\Vert Y \Vert_{\mathbb S^2_d(\F^0)}^2 + 4L^2\Vert Z\Vert_{\H^2_d(\F^0)} + 2 \Vert f(\cdot,0,0) \Vert_{\H^2_d(\F^0)}^2
    \end{align*}
    and
    \begin{align*}
      &\esp{(\cK^n_T)^2} \le \esp{|K_T|^2}.
    \end{align*}
    Thus, by these estimates and the fact that $A^\phi_T - A^\phi_t \in L^2(\cF^\infty_T)$ as $\phi$ is admissible, there exists a constant $\Lambda_1 > 0$ such that
    \begin{align}
      \label{est A} \esp{|\cA^n_{\tau_n \wedge T}|^2} \le \Lambda_1 + \Lambda \esp{\int_t^{\tau_n\wedge T} \ud[\cM^n]_s}.
      \end{align}
      Using \eqref{eq apres t} applied to $\phi^n$ and Itô's formula between $t$ and $\tau_n \wedge T$, since $\cM^n$ is a square integrable martingale orthogonal to $W$ and $A^n, \cA^n, \cK^n$ are nondecreasing and nonnegative, we get
      \begin{align}
        &\nonumber \esp{|\cY^n_t|^2 + \int_t^{\tau_n \wedge T} |\cZ^n_s|^2 \ud s + \int_t^{\tau_n \wedge T} \ud [\cM^n]_s} \\
        &\nonumber = \esp{|\cY^n_{\tau_n \wedge T}|^2 + 2\int_t^{\tau_n \wedge T} \cY^n_s f^{a^n_s}(s,\cY^n_s,\cZ^n_s)\ud s - 2\int_t^{\tau_n \wedge T} \cY^n_s \ud A^n_s \right.\\
        &\nonumber \hspace{.7cm} \left. + 2\int_t^{\tau_n \wedge T} \cY^n_s \ud \cA^n_s + 2\int_t^{\tau_n \wedge T} \cY^n_s \ud \cK^n_s} \\
        &\nonumber \le \esp{|\cY^n_{\tau_n \wedge T}|^2} + 2\esp{\int_t^{\tau_n \wedge T} |\cY^n_s f^{a^n_s}(s,\cY^n_s,\cZ^n_s)| \ud s} + 2 \esp{\int_t^{\tau_n \wedge T} |\cY^n_s| \ud A^n_s}\\
        &\label{est M} \hspace{.2cm}+ 2 \esp{\int_t^{\tau_n \wedge T} |\cY^n_s| \ud \cA^n_s} + 2 \esp{\int_t^{\tau_n \wedge T} |\cY^n_s| \ud \cK^n_s}.
      \end{align}
    We have, using Young's inequality, for some $\epsilon > 0$, and \eqref{est A},
    \begin{align*}
      &\esp{\int_t^{\tau_n \wedge T} |\cY^n_s f^{a^n_s}(s,\cY^n_s,\cZ^n_s)| \ud s} \le \frac12 \esp{\int_t^{\tau_n \wedge T} |\cY^n_s|^2 \ud s} + \frac12 \esp{\int_t^{\tau_n \wedge T} |f^{a^n_s}(s,\cY^n_s,\cZ^n_s)|^2\ud s} \\
                                         &\hspace{5.25cm} \le T(\frac12 + 2L^2)\Vert Y \Vert_{\mathbb S^2_d(\F^0)}^2 + 2L^2\Vert Z\Vert_{\H^2_d(\F^0)} + \Vert f(\cdot,0,0) \Vert_{\H^2_d(\F^0)}^2,\\
                                         &\esp{\int_t^{\tau_n \wedge T} |\cY^n_s| \ud A^n_s} \le \frac{1}{2} \Vert Y \Vert^2_{\mathbb S^2_d(\F^0)} + \frac{1}{2}\esp{(A^\phi_T - A^\phi_t)^2}, \\
                                         &\esp{\int_t^{\tau_n \wedge T} |\cY^n_s| \ud \cK^n_s} \le \frac{1}{2} \Vert Y \Vert^2_{\mathbb S^2_d(\F^0)} + \frac{1}{2}\esp{|K_T|^2}, \\
    \end{align*}
    and
    \begin{align*}
      \esp{\int_t^{\tau_n \wedge T} |\cY^n_s| \ud \cA^n_s} &\le \frac{1}{2\epsilon} \Vert Y \Vert^2_{\mathbb S^2_d(\F^0)} + \frac{\epsilon}{2}\esp{(\cA^n_{\tau_n \wedge T})^2} \\
                                         &\le \frac{1}{2\epsilon} \Vert Y \Vert^2_{\mathbb S^2_d(\F^0)} + \frac{\epsilon}{2}\left(\Lambda_1 + \Lambda \esp{\int_t^{\tau_n\wedge T} \ud[\cM^n]_s}\right).
    \end{align*}
    Using these estimates together with \eqref{est M} gives, for a constant $C_\epsilon > 0$ independent of $n$,
    \begin{align}
      \left(1 - \epsilon \Lambda\right)\esp{\int_t^{\tau_n \wedge T} \ud [\cM^n]_s} \le C_\epsilon \left(\Vert Y \Vert^2_{\mathbb S^2_d(\F^0)}+\Vert Z\Vert_{\H^2_{d\times \kappa}(\F^0)} + \Vert f(\cdot,0,0) \Vert_{\H^2_d(\F^0)}^2+\esp{|K_T|^2}\right),
    \end{align}
    and chosing $\epsilon = \frac{1}{2\Lambda}$ gives that $\esp{\int_t^{\tau_n \wedge T} \ud [M^n]_s}$ is upper bounded independently of $n$. We also get an upper bound independent of $n$ for $\esp{(\cA^n_{\tau_n \wedge T})^2}$ by \eqref{est A}.\\
    Since $\int_t^{\tau_n \wedge T} \ud [\cM^n]_s$ (resp. $|\cA^n_{\tau_n \wedge T}|^2$) is nondecreasing to $\int_t^T \ud [\cM^\phi]_s$ (resp. to $|\cA^\phi_T|^2$), we obtain by monotone convergence the first part of Lemma \eqref{lem estimes startegies gene}. It is also clear that $\cM^\phi$ is a martingale satisfying $\cM^\phi_t = 0$.\\
    We now prove \eqref{mart en t}. Using that $\espcond{(N^\phi_t)^2}{\cF^0_t}$ is almost-surely finite as $\phi$ is admissible and $\hat{c}>0$, we compute,
    \begin{align}
      &\nonumber \espcond{\left(\sum_{k \ge 0} \left(Y^{\zeta_{k+1}}_{\tau_{k+1}} - \espcond{Y^{\zeta_{k+1}}_{\tau_{k+1}}}{\cF^k_{\tau_{k+1}}}\right)\ind{\tau_{k+1} \le t}\right)^2}{\cF^0_t}\\
      &\nonumber \le \espcond{\left(\sum_{k \ge 0} \left|Y^{\zeta_{k+1}}_{\tau_{k+1}} - \sum_{j=1}^d p^{\alpha_{k+1}}_{\zeta_k, j} Y^j_t\right|\ind{\tau_{k+1} \le t}\right)^2}{\cF^0_t} \\
      &\le 4 |Y_t|^2 \espcond{(N^\phi_t)^2}{\cF^0_t} < +\infty\quad \text{a.s.}
    \end{align}
 \eproof

\subsubsection{An optimal strategy}
\textcolor{black}{
We now introduce a strategy, which turns out to be optimal for the control problem. This strategy is the natural extension to our setting of the optimal one for \emph{classical switching problem}, see e.g. \cite{HT10}. The first key step is to prove that this strategy is admissible, which is more involved than in the classical case due to the randomisation.}

\vspace{2mm}
\noindent Let $\phi^\star = (\zeta^\star_0, (\tau^\star_n)_{n \ge 0}, (\alpha^\star_n)_{n \ge 1})$ defined by $\tau^\star_0 = t$ and $\zeta^\star_0 = i$ and inductively by:
\begin{align}
  \label{def tau star} \tau^\star_{k+1} &= \inf \left\{ \tau^\star_k \le s \le T : Y^{\zeta^\star_k}_s = {\max}_{u \in \cC} \left\{  \sum_{j=1}^d P^u_{\zeta_k,j} Y^j_s - \bar{c}_{\zeta^\star_k}^{u} \right\} \right\} \wedge (T + 1), \\
  \label{def alpha star} \alpha^\star_{k+1} &= {\min}\left\{\alpha \in \argmax_{u \in \cC} \left\{ \sum_{j=1}^d P^u_{\zeta_k,j} Y^j_{\tau^\star_{k+1}} - \bar{c}_{\zeta_k}^{u} \right\}\right\},
\end{align}
recall that $\cC$ is ordered.\\
In the following lemma, we show that, since $\cD$ has non-empty interior, the number of switch (hence the cost) required to leave any point on the boundary of $\cD$ is square integrable, following the strategy $\phi^\star$. This result will be used to prove that the cost associated to $\phi^\star$ satisfies $\espcond{\left(A^{\phi^\star}_t\right)^2}{\cF_t^0}$ is almost surely finite.

 \begin{Lemma}
  \label{lem sortie bord}
Let Assumption \ref{exist}-ii) hold.
  For $y \in \cD$, we define
    \begin{align}
      S(y) = \left\{1 \le i \le d : y_i = \max_{u \in \cC} \left\{ \sum_{j=1}^d P^u_{i,j} y_j - \bar{c}_{i}^{u} \right\} \right\},
    \end{align}
    and $(u_i)_{i \in S(y)}$ the family of elements of $\cC$ given by
    $$u_i = \min  \argmax_{u \in \cC} \left\{ \sum_{j=1}^d P^{u}_{i,j} y_j - \bar{c}_{i}^{u}\right\}.$$
    Consider the homogeneous Markov Chain $X$ on $S(y) \cup \{0\}$ defined by, for $k \ge 0$ and $i,j \in S(y)^2$,
    \begin{align*}
      \P(X_{k+1} = j | X_k = i) &= P^{u_i}_{i,j}, \\
      \P(X_{k+1} = 0 | X_k = i) &= 1 - \sum_{j \in S(y)} P^{u_i}_{i,j},\\
      \P(X_{k+1}=0 | X_k=0) &= 1,\\
      \P(X_{k+1}=i | X_k=0) &=0.
    \end{align*}
    Then $0$ is accessible from every $i \in S(y)$, meaning that $X$ is an absorbing Markov Chain.\\
    Moreover, let $N(y) = \inf \{n \ge 0 : X_n = 0 \}$. Then $N(y) \in L^2(\P^i)$ for all $i \in S(y)$, where $\P^i$ is the probability satisfying $\P^i(X_0 = i) = 1$.
  \end{Lemma}
  \begin{proof}
    Assume that there exists $i \in S(y)$ from which $0$ is not accessible. Then every communicating class accessible from $i$ is included in $S(y)$. In particular, there exists a recurrent class $S' \subset S(y)$. For all $i \in S'$, we have $P^{u_i}_{i,j} = 0$ if $j \not\in S'$ since $S'$ is recurrent. Moreover, since $S' \subset S(y)$, we obtain, for all  $i \in S'$, by definition of $S(y)$,
    \begin{align}
    \label{egalites y}
      y_i = \sum_{j \in S'} P^{u_i}_{i,j} y_j - \bar{c}_{i}^{u_i}.
    \end{align}
    Since $S'$ is a recurrent class, the matrix $\tilde{P} = (P^{u_i}_{i,j})_{i,j \in S'}$ is stochastic and irreducible.\\
    By definition of $\cD$, we have 
    $$\cD \subset \R^{d - |S'|} \times \left\{z \in \R^{|S'|} | z_i \ge \sum_{j \in S'} P^{u_i}_{i,j} z_j - \bar{c}_{i}^{u_i}, i \in S' \right\} = \R^{d - |S'|} \times \cD'.$$ 
    With a slight abuse of notation, we do not renumber coordinates of vectors in $\cD'$.\\
    Let $i_0 \in S'$ and let us restrict ourself to the domain $\cD'$. According to Lemma \ref{lem-general}, $\cD'$ is invariant by translation along the vector $(1,...,1)$ of $\R^{|S'|}$. Moreover, Assumption \ref{ass-uncontrolled-irred} is fulfilled since $\tilde{P}$ is irreducible and controls $(u_i)_{i \in S(y)}$ are set. So, Proposition \ref{pr necessary conditions} gives us that $\cD' \cap \{z \in \R^{|S'|} | z_{i_0}=0\}$ is a compact convexe polytope. Recalling \eqref{egalites y}, we see that 
    $(y_i - y_{i_0})_{i \in S'}$ is a point of $\cD' \cap \{z \in \R^{|S'|} | z_{i_0}=0\}$ that saturates all the inequalities. So, $(y_i - y_{i_0})_{i \in S'}$ is an extreme points of $\cD' \cap \{z \in \R^{|S'|} | z_{i_0}=0\}$ and all extreme points are given by 
    $$\mathcal{E}:= \left\{z \in \R^{|S'|} | z_i = \sum_{j \in S'} P^{u_i}_{i,j} z_j - \bar{c}_{i}^{u_i}, i \in S', z_{i_0}=0 \right\}.$$
    Recalling that $\cD'$ is compact,  $\mathcal{E}$ is a nonempty bounded affine subspace of $\R^{|S'|}$, so it is a singleton. Since $\cD' \cap \{z_{i_0}=0\}$ is a compact convex polytope, it is the convex hull of $\mathcal{E}$ and so it is also a singleton.
     Hence $\cD'$ is a line in $\R^{|S'|}$. Moreover, $|S'| \ge 2$ as $P^u_{i,i} \neq 1$ for all $u \in \cC$ and $i \in \{1,\dots,d\}$. Thus $\cD \subset \R^{d - |S'|} \times \cD'$ gives a contradiction with the fact that $\cD$ has non-empty interior and the first part of the lemma is proved.\\
    %
   %
    Finally, we have $N(y) \in L^2(\P^i)$ for all $i \in S(y)$ thanks to Theorem 3.3.5 in \cite{KS81}.
  \eproof
\end{proof}

\begin{Lemma} \label{lem-integ-adm}
  Assume that assumption \eqref{exist} is satisfied.
  The strategy $\phi^\star$ is admissible.
\end{Lemma}
  \proof
   For $n \ge 1$, we consider the admissible strategy $\phi_n = (\zeta_0, (\tau^n_k)_{k \ge 0}, (\alpha^n_k)_{k \ge 1})$ defined by $\zeta^n_0 = i = \zeta_0^\star, \tau^n_k = \tau^\star_k, \alpha^n_k = \alpha^\star_k$ for $k \le n$, and $\tau^n_k = T+1$ for all $k > n$. We set $\cY^n_s := \cY^{\phi_n}_s, \cZ^n_s := \cZ^{\phi_n}_s$ and so on, for all $s \in [t,T]$.\\
    By definition of $\tau^\star, \alpha^\star$, recall \eqref{def tau star}-\eqref{def alpha star}, it is clear that $\cA^n_{s\wedge \tau_n^\star} = 0$ and that $\int_{\tau^\star_k \wedge s}^{\tau^\star_{k+1} \wedge s} \ud K^{\zeta^\star_k}_u = 0$ for all $k < n$ and $s \in [t,T]$. The identity \eqref{eq apres t} for the admissible strategy $\phi^n$ gives
    \begin{align*}
      \cY^n_t = \cY^n_{\tau^\star_n \wedge T} + \int_t^{\tau^\star_n \wedge T} f^{a^n_s}(s,\cY^n_s,\cZ^n_s)\ud s - \int_t^{\tau^\star_n \wedge T} \cZ^n_s \ud W_s - \int_t^{\tau^\star_n \wedge T} \ud \cM^n_s - \int_t^{\tau^\star_n \wedge T} \ud A^n_s.
    \end{align*}
    Using similar arguments and estimates as in the precedent proof, we get
    \begin{align}
      \label{est A opti} \esp{|A^n_{\tau^\star_n \wedge T} - A^n_t|^2} \le \Lambda_1 + \Lambda\esp{\int_t^{\tau^\star_n \wedge T} \ud[\cM^u]_s},
    \end{align}
    and, for $\epsilon > 0$,
    \begin{align}
      \left(1 - \epsilon\Lambda\right)\esp{\int_t^{\tau^\star_n \wedge T} \ud [\cM^n]_s} \le C_\epsilon \left(\Vert Y \Vert^2_{\mathbb S^2_d(\F^0)}+\Vert Z\Vert_{\H^2_{d\times \kappa}(\F^0)} + \Vert f(\cdot,0,0) \Vert_{\H^2_d(\F^0)}^2\right).
    \end{align}
    Choosing $\epsilon = \frac{1}{2\Lambda}$ gives that $\esp{\int_t^{\tau^\star_n \wedge T} \ud [\cM^n]_s}$ and $\esp{|A^n_{\tau^\star_n \wedge T} - A^n_t|^2}$ are upper bounded uniformly in $n$, hence by monotone convergence, we get that $A^{\phi^\star}_T - A^{\phi^\star}_t \in L^2(\cF^\infty_T)$.\\
    It remains to prove that $A^{\phi^\star}_t \in L^2(\cF^0_t)$. We have $A^{\phi^\star}_t \le \check{c} N^{\phi^\star}_t$, and $\espcond{(N^{\phi^\star}_t)^2}{\cF^0_t} < +\infty$ a.s. is immediate {from Lemma \ref{lem sortie bord}}, since $\espcond{(N^{\phi^\star}_t)^2}{\cF^0_t} = \Psi(Y_t)$ with $\Psi(y) = \mathbb E^i\left[(N(y))^2\right], y \in \cD$, where $\mathbb E^i$ is the expectation under the probability $\P^i$ defined in Lemma \ref{lem sortie bord}. \eproof

\subsubsection{Proof of Theorem \ref{th representation result}}
We now have all the key ingredients to conclude the proof of Theorem \ref{th representation result}.

\noindent 1. Let $\phi \in \mathscr A^i_t$, and consider the identity \eqref{eq apres t}. Since $\cM^\phi$ is a square integrable martingale, orthogonal to $W$, and since $\cA^\phi_T + \cK^\phi_T \in L^2(\cF^\infty_T)$ and the process $\cA^\phi + \cK^\phi$ is nonnegative and nondecreasing, the comparison Theorem \ref{comparison} gives $\cY^\phi_t \ge U^\phi_t$, recall \eqref{switched BSDE}.\\
    Now, we have
    \begin{align}
      \nonumber \cY^\phi_t &= Y^i_t + \sum_{k \ge 0} \left( Y^{\zeta_{k+1}}_t - Y^{\zeta_k}_t \right) \ind{\tau_{k+1} \le t} \\
      \nonumber &= Y^i_t + \sum_{k \ge 0} \left(Y^{\zeta_{k+1}}_t - \espcond{Y^{\zeta_{k+1}}_t}{\cF^k_{\tau_{k+1}}} \right) \ind{\tau_{k+1} \le t}\\
                           &\label{eq en t} \hspace{1cm} - \sum_{k \ge 0} \left(Y^{\zeta_k}_t - \sum_{j=1}^d P^{\alpha_{k+1}}_{\zeta_{k,j}} Y^j_{\tau_{k+1}} + \bar{c}^{\alpha_{k+1}}_{\zeta_k}\right)\ind{\tau_{k+1}\le t} + A^\phi_t.
    \end{align}
    Since $U^\phi_t \le \cY^\phi_t$ and $\sum_{k \ge 0} \left(Y^{\zeta_k}_t - \sum_{j=1}^d P^{\alpha_{k+1}}_{\zeta_{k,j}} Y^j_{\tau_{k+1}} + \bar{c}^{\alpha_{k+1}}_{\zeta_k}\right)\ind{\tau_{k+1}\le t} \ge 0$, we get
    \begin{align}
      \nonumber U^\phi_t - A^\phi_t &\le Y^i_t + \sum_{k \ge 0} \left(Y^{\zeta_{k+1}}_t - \espcond{Y^{\zeta_{k+1}}_t}{\cF^k_{\tau_{k+1}}} \right) \ind{\tau_{k+1} \le t}\\
                                    &\nonumber \hspace{1cm}- \sum_{k \ge 0} \left(Y^{\zeta_k}_t - \sum_{j=1}^d P^{\alpha_{k+1}}_{\zeta_{k,j}} Y^j_{\tau_{k+1}} + \bar{c}^{\alpha_{k+1}}_{\zeta_k}\right)\ind{\tau_{k+1}\le t} \\
                                    &\le Y^i_t + \sum_{k \ge 0} \left(Y^{\zeta_{k+1}}_t - \espcond{Y^{\zeta_{k+1}}_t}{\cF^k_{\tau_{k+1}}} \right) \ind{\tau_{k+1} \le t}.
    \end{align}
    Using \eqref{mart en t}, we can take conditional expectation on both side with respect to $\cF^0_t$ to obtain the result.\\
2. Lemma \ref{lem-integ-adm} shows that the strategy $\phi^\star$ is admissible. Using \eqref{eq apres t}, since $\cA^{\phi^\star} = 0$ and $\int_{\tau^\star_k \wedge T}^{\tau^\star_{k+1} \wedge T} \ud K^{\zeta^\star_k}_u = 0$ for all $k \ge 0$, we obtain
    \begin{align}
      \cY^{\phi^\star}_s = \xi^{a^\star_T} + \int_s^T f^{a^\star_u}(u,\cY^{\phi^\star}_u,\cZ^{\phi^\star}_u)\ud u - \int_s^T \cZ^{\phi^\star}_u \ud W_u - \int_s^T \ud \cM^{\phi^\star}_u - \int_s^T \ud A^{\phi^\star}_u.
    \end{align}
    By uniqueness Theorem \ref{ex-uni-bsde}, we get that $\cY^{\phi^\star}_t = U^{\phi^\star}_t$, recall \eqref{switched BSDE}.\\
    We also have
    \begin{align*}
      \cY^{\phi^\star}_t &= Y^i_t + \sum_{k \ge 0} \left( Y^{\zeta^\star_{k+1}}_t - Y^{\zeta^\star_k}_t \right) \ind{\tau^\star_{k+1} \le t} \\
                         &= Y^i_t + \cM^{\phi^\star}_t + A^{\phi^\star}_t,
    \end{align*}
    thus $U^{\phi^\star}_t - A^{\phi^\star}_t = \cY^{\phi^\star}_t - A^{\phi^\star}_t = Y^i_t + \cM^{\phi^\star}_t$, and taking conditional expectation gives the result. \eproof
\section{Obliquely Reflected BSDEs associated to randomised switching problems} \label{existence}

In this section, we study the Obliquely Reflected BSDE \eqref{orbsde}-\eqref{orbsde2}-\eqref{orbsde3} associated to the \emph{switching problem with controlled randomisation}. We address the question of existence of such BSDEs. Indeed, as observed in the previous section, under appropriate assumptions, uniqueness follows directly from the control problem representation, see Corollary \ref{co uniqueness RBSDE} and Proposition \ref{prop unicite couts generaux}. We first give some general properties of the domain $\cD$ and identify necessary and sufficient conditions linked to the non-emptiness of its interior. The non-empty interior property is key for our existence result and is not trivially obtained in the setting of signed costs. This is mainly the purpose of Section \ref{section prop du domaine de reflection}.
Then, we prove existence results for the associated BSDE in the  Markovian framework, in Section \ref{sub se markov}, and in the non-Markovian framework, in Section \ref{sub se non markov}, relying on the approach in \cite{chassagneux2018obliquely}. Existence results in \cite{chassagneux2018obliquely} are obtained for general obliquely reflected BSDEs where the oblique reflection is specified through an operator $H$ that transforms, on the boundary of the domain, the normal cone into the oblique direction of reflection. Thus, the main difficulty is to construct this operator $H$ with some specific properties needed to apply the existence theorems of \cite{chassagneux2018obliquely}. 
This task is carried out successfully for the \emph{randomised switching problem} in the Markovian 
framework. We also consider an example of \emph{switching problem with controlled randomisation} in this framework. In the non-Markovian framework, which is more challenging as more properties are required on $H$, we prove the well-posedness of the BSDE for some examples of \emph{randomised switching problem}.

\subsection{Properties of the domain of reflection }
\label{section prop du domaine de reflection}
In this section, we study the domain where the solution of the reflected BSDEs is constrained to take its values. The first result  shows that the domain $\cD$ defined in \eqref{domain} is invariant by translation along the vector $(1,\dots,1)$ and deduces some property for its normal cone. Most of the time, we will thus be able to limit our  study to
\begin{align}\label{eq de D slice}
\Dslice = \cD \cap \set{ y \in \R^d \,|\, y_d = 0}\;.
\end{align}

\begin{Lemma} \label{lem-general} 
For all $x \in \cD$, we have
  \begin{enumerate}
  \item \textcolor{black}{$x+h\sum_{i=1}^d e_i \in \cD$, for all $h \in \mathbb{R}$,}
  \item there is a unique decomposition $x = y^x + z^x$ with $y^x \in \Dslice$ and $z^x \in \R \left(\sum_{i=1}^d e_i\right)$,
  \item if $x \in \cD$, we have $\cC(x) \subset \{v \in \R^d : \sum_{i=1}^d v_i = 0\}$,
  \item $\cC(x) = \cC(y_x)$, where $y_x$ is from the above decomposition.
  \end{enumerate}
\end{Lemma}
\begin{proof}
 \textcolor{black}{1. If $i \in \{1,\dots,d\}$, we have
    \begin{align}
      \nonumber x_i-h =  &\ge \max_{u \in \cC} \left(\sum_{j=1}^d P^u_{i,j} x_j - c_{i}^{u}\right) - h = \max_{u \in \cC} \left( \sum_{j=1}^d P^u_{i,j} (x_j - h) - c_{i}^{u}\right),
    \end{align}
    and thus $x+h\sum_{i=1}^d e_i \in \cD$.\\    
2.  We set $y^x = x - z^x$ with $z^x = x_d \sum_{i=1}^d e_i$. It is clear that $y^x_d = 0$, and $y^x \in \cD$ thanks to the first point. 
   The uniqueness is clear since we have necessarily $z^x = x_d \sum_{i=1}^d e_i$.\\
3.  Point 1. shows that $x \pm \sum_{i=1}^d e_i \in \cD$. Let $v \in \cC(x)$. }Then we have, by definition,
    \begin{align*}
      0 &\ge v^\top (x \pm \sum_{i=1}^d e_i - x) = \pm v^\top \sum_{i=1}^d e_i = \pm \sum_{i=1}^d v_i,
    \end{align*}
    and thus, $\sum_{i=1}^d v_i = 0$.\\
4. Let $x \in \cD$. Since $x = y^x + x_d \sum_{i=1}^d e_i$, it is enough to show that for all $w \in \cD$ and all $a \in \R, \cC(w) \subset \cC(w + a \sum_{i=1}^d e_i)$.\\
    Let $v \in \cC(w)$. We have, for all $z \in \cD$, since $\sum_{i=1}^d v_i = 0$ and $v^\top(z-w) \le 0$,
    \begin{align*}
      v^\top(z-(w+a\sum_{i=1}^d e_i)) &= v^\top(z-w) - a v^\top \sum_{i=1}^d e_i = v^\top(z-w) \le 0,
    \end{align*}
    and thus $v \in \cC(w+a\sum_{i=1}^d e_i)$. \eproof
\end{proof}

\vspace{2mm}

\noindent Before studying the domain of reflection, we introduce three examples in dimension $3$ of switching problems. On Figure  \ref{figure exemples domaines}, we draw the domain $\Dslice$ for these three different switching problems to illustrate the impact of the various controlled randomisations on the shape of the reflecting domain.
\begin{description}
 \item[Example 1:] Classical switching problem with a constant cost $1$, i.e. $\cC=\{1,2\}$,
 \begin{align*}
P^1 = 
\left(
\begin{array}{ccc}
0 & 1  & 0 \\
0  & 0 & 1 \\
1 & 0 & 0 
\end{array}
\right)
\;,\;
P^2 = 
\left(
\begin{array}{ccc}
0 & 0  & 1 \\
1  & 0 & 0 \\
0 & 1 & 0 
\end{array}
\right)
\;,\;
\bar{c}^1 =\left(
\begin{array}{c}
1  \\
1  \\
1 \\
\end{array}
\right)
\;
\text{ and }
\;
\bar{c}^2=\left(
\begin{array}{c}
1  \\
1   \\
1 \\
\end{array}
\right).
\end{align*}
 \item[Example 2:] Randomised switching problem with $\cC=\{0\}$,
 \begin{align*}
P^0 = 
\left(
\begin{array}{ccc}
0 & 1/2  & 1/2 \\
1/2  & 0 & 1/2 \\
1/2 & 1/2 & 0 
\end{array}
\right)
\;\text{ and } \;
\bar{c}^0 =\left(
\begin{array}{c}
1  \\
1  \\
1 \\
\end{array}
\right).
\end{align*}
 \item[Example 3:] Switching problem with controlled randomisation where $\cC=[0,1]$,  
\begin{equation}
\label{transitions et couts exemple controle} 
 P^u=    \left(\begin{array}{ccc}
            0 & u & 1-u\\
            1-u & 0 & u\\
            u & 1-u & 0\\
          \end{array}\right) \;\text{ and } \;
 \bar{c}^0 =\left(
\begin{array}{c}
1-u(1-u)  \\
1-u(1-u)  \\
1-u(1-u) \\
\end{array}
\right) \quad \forall u \in [0,1].
\end{equation}
In this example, the transitions matrix are given by convex combinations of transitions matrix  of Example 1.
\end{description}

\begin{figure}[htp]
 \centering
 \begin{center}
 \includegraphics[width=10cm]{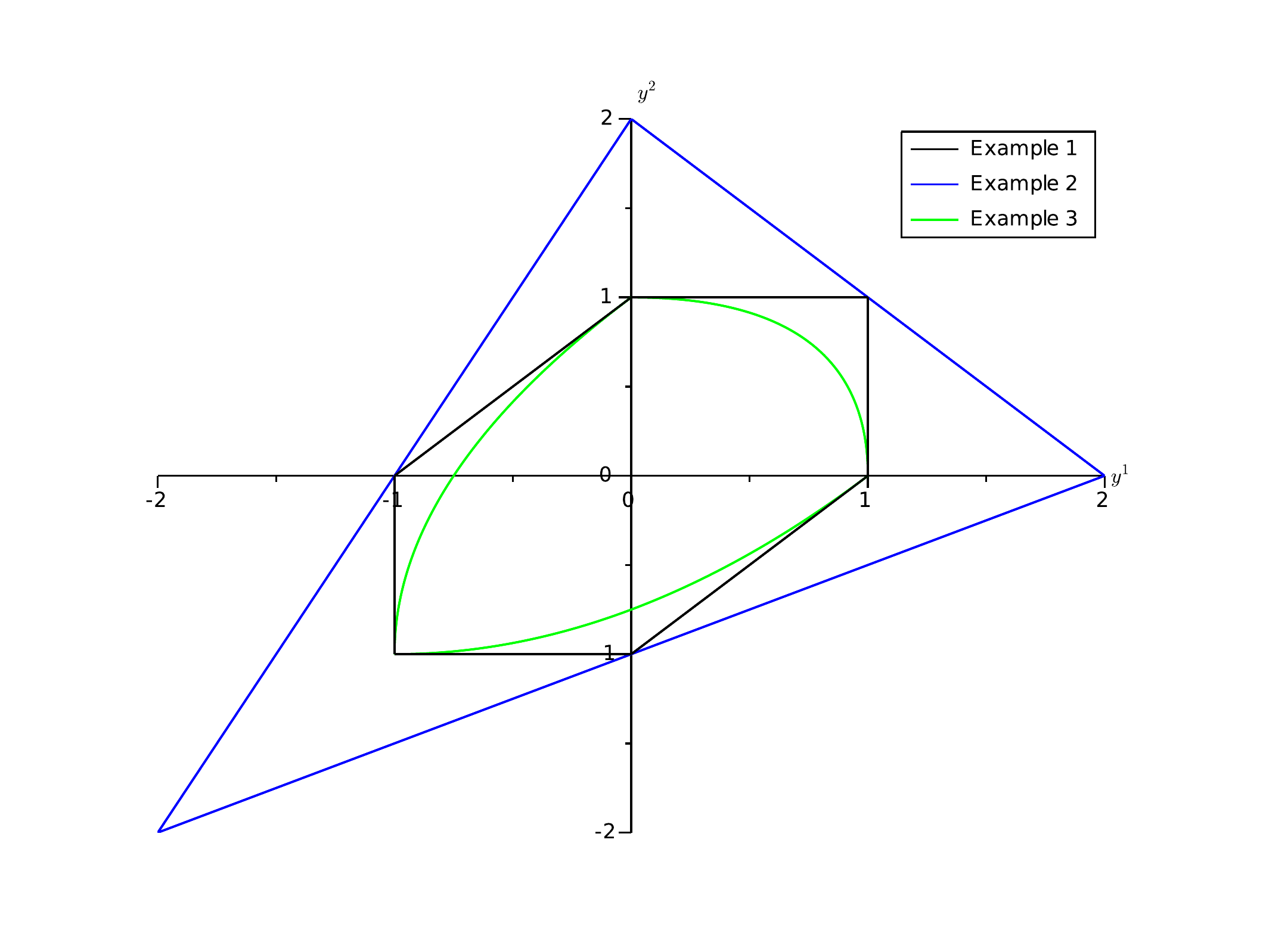}
\end{center}
 \caption{Domaine $\Dslice$ for three examples of switching problems with or without controlled randomisation.}
 \label{figure exemples domaines}
\end{figure}

\textcolor{black}{
\begin{Remark}
 For the randomised switching problem, in any dimension, we can replace $(P_{i,j})_{1 \le j \le d}$ by $\left( \frac{P_{i,j}}{1-P_{i,i}}\1_{i \neq j}\right)_{1 \le j \le d}$ and $\bar{c}_i$ by $\frac{\bar{c}_i}{1-P_{i,i}}$ as soon as $P_{i,i}<1$, without changing $\cD$. The factor $({1-P_{i,i}})^{-1}$ in the cost has to be seen as the expectation of the geometric law of the number of trials needed to exit state $i$. So assuming that diagonal terms are zero is equivalent to assume that $P_{i,i}<1$, for all $1 \le i \le d$.
\end{Remark}
}

\subsubsection{The uncontrolled case}

In this part, we study the domain $\cD$ for a fixed control, which is fixed to be $0$, without loss of generality.
The properties of the domain are closely linked in this case to the homogeneous Markov chain, denoted $X$, associated to the stochastic matrix $P$. For this part, we thus work with the following assumption.
\begin{Assumption} \label{ass-uncontrolled-irred}
  The set of control is reduced to $\cC = \{0\}$.
  The Markov chain $X$ with stochastic matrix $P = (P_{i,j})_{1 \le i,j \le d} := (P^0_{i,j})_{1 \le i,j \le d}$ is irreducible.
\end{Assumption}
\noindent Our main goal is to find necessary and sufficient conditions to characterize the non-emptiness of the domain $\cD$. To this end, we will introduce some quantities related to the Markov Chain $X$ and the costs vector $\bar{c}:=\bar{c}^0$.  

\vspace{2mm}
\noindent For $1 \le i,j \le d$, we consider  the  expected cost along an ``excursion'' from state $i$ to $j$:
\begin{align} \label{eq de expected cost}
\bar{C}_{i,j} := \esp{\sum_{n=0}^{\tau_j - 1} \bar{c}_{X_n} \Big| X_0 = i}
\textcolor{black}{=\esp{\sum_{n=0}^{\tau_j-1}c_{X_n,X_{n+1}} \Big | X_0 = i}}
\end{align}
where
\begin{align*}
\tau_j := \inf \set{ t \ge 1\,|\, X_t = j}\;.
\end{align*}
We also define 
\begin{align}  \label{eq de C}
C_{j,j} := 0 \; \text{ and } \; C_{i,j} = \bar{C}_{i,j} \text{ for } 1 \le i \neq j \le d \;.
\end{align}
We observe that, introducing $\tilde{\tau}_j := \inf \set{ t \ge 0\,|\, X_t = j}$, the cost $C$ rewrites  as $\bar{C}$:
\begin{align*}
{C}_{i,j} := \esp{\sum_{n=0}^{\tilde{\tau}_j } \bar{c}_{X_n}\1_{\set{X_n \neq j}} \Big| X_0 = i}\;, \text{ for } 1 \le i,j \le d\,.
\end{align*}
\textcolor{black}{Let us remark that $\esp{\tau+\tilde{\tau}}<+\infty$ and so $\bar{C}$ and $C$ are finite since the Markov chain is irreducible recurrent.} 
\\
\noindent Setting  $Q = I_d - P$,  the domain $\cD$, defined in \eqref{domain}, rewrites:
\begin{align} \label{eq de domain with Q}
  \cD 
  & = \{x \in \R^d : Qx + c \succcurlyeq 0\}.
\end{align}
Since $P$ is irreducible, it is well known (see for example \cite{B15}, Section 2.5) that for all $1 \le i, j \le d$, the matrix $Q^{(i,j)}$ is invertible, and that we have 
\begin{align}\label{eq inv meas}
\tilde{\mu}_i := \det Q^{(i,i)} = (-1)^{i+j} \det Q^{(i,j)}> 0 .
\end{align}
Moreover, 
$\tilde{\mu} Q = 0$ with $\tilde\mu = (\tilde\mu_i)_{i=1}^d$, i.e. ${\mu} := \frac{\tilde\mu}{\sum_{i=1}^d \tilde\mu_i}$ is the unique invariant probability measure for  the Markov chain with transition matrix $P$. 

\noindent We now obtain some necessary conditions for the domain to be non-empty.
Let us first observe that
\begin{Lemma}
The mean costs $C$ are given for $1 \le i \neq j \le d$ by
\begin{align}\label{eq carac C}
C_{i,j} = \left((Q^{(j,j)})^{-1} \bar{c}^{(j)}\right)_{i-\1_{\set{i>j}} } \;.
\end{align}
\end{Lemma}
\proof
1. We first show that for $1 \le i,j \le d$:
\begin{align}\label{eq dpp bar C}
\bar{C}_{i,j} = \bar{c}_i + \sum_{\ell \neq j} \bar{C}_{\ell, j}P_{i,\ell} \;.
\end{align}
From \eqref{eq de expected cost}, we have
\begin{align*}
\bar{C}_{i,j} &= \esp{\sum_{n=0}^{+\infty} \bar{c}_{X_n} \1_{\set{n<\tau_j}}\Big|X_0 = i} = \bar{c}_i + \esp{\sum_{n=1}^{+\infty} \bar{c}_{X_n} \1_{\set{n<\tau_j}}\Big|X_0 = i}.
\end{align*}
Then, since for all $n \ge 1$, $\set{X_1 = j} \cap \set{n<\tau_j} = \emptyset$, we get
\begin{align*}
\bar{C}_{i,j} &=\bar{c}_i + \esp{\sum_{n=1}^{+\infty} \sum_{\ell \neq j} \bar{c}_{X_n}\1_{\set{X_1 = \ell}} \1_{\set{n<\tau_j}}\Big|X_0 = i}.
\end{align*}
We compute that, for $\ell \neq j$,
\begin{align*}
\esp{\sum_{n=1}^{+\infty}  \bar{c}_{X_n}\1_{\set{X_1 = \ell}} \1_{\set{n<\tau_j}}\Big|X_0 = i}
= \esp{\sum_{n=1}^{+\infty}  \bar{c}_{X_n} \1_{\set{n<\tau_j}}\Big|X_1 = \ell}P_{i,\ell}.
\end{align*}
The proof of  \eqref{eq dpp bar C} is then concluded observing that from the Markov property,
$$\esp{\sum_{n=1}^{+\infty}  \bar{c}_{X_n} \1_{\set{n<\tau_j}}\Big|X_1 = \ell} = \bar{C}_{\ell, j}\,.$$
2. From \eqref{eq dpp bar C}, we deduce, recall Definition \eqref{eq de C}, that, for $i \neq j$,
\begin{align}\label{eq dpp C}
{C_{i,j}} = \bar{c}_i + \sum_{\ell \neq j} {C}_{\ell,j}P_{i,\ell} \;.
\end{align}
This equality simply rewrites $Q^{(j,j)} C_{\cdot,j} = \bar{c}^{(j)}$, which concludes the proof.
\eproof 



\vspace{2mm}
\begin{Proposition}\label{pr necessary conditions} Assume $\cD$ is non-empty.  Then,
\begin{enumerate}
\item the mean cost with respect to the invariant measure is non-negative, namely:
\begin{align}\label{eq cond mean cost non empty}
{\mu} \bar{c} \ge 0.
\end{align}
\item For all $1 \le i,j \le d$,
\begin{align}\label{eq cond C non empty}
\min_{1 \le i,j \le d} \left(C_{i,j} + C_{j,i} \right) \ge 0.
\end{align}
\item The set $\Dslice$ is compact in $\{y \in \R^d | y_d = 0\}$.
\end{enumerate}
Moreover, if $\cD$ has non-empty interior, then
\begin{align}\label{eq cond cost non empty interior}
\mu \bar{c} > 0\;\text{ and } \;\min_{1 \le i \neq j \le d} \left(C_{i,j} + C_{j,i} \right) > 0\;.
\end{align}
\end{Proposition}

\proof
1.a We first show the key relation:
  \begin{align} \label{eq encadrement}
    - C_{i,j} \le x_i - x_j \le C_{j,i} \;,\; \text{ for } 1 \le i,j \le d.
  \end{align}
  For $j \in \{1,\dots,d\}$ and $x \in \R^d$, we introduce $\pi^j(x) \in \R^{d-1}$, given by,
  $$\pi^j(x)_k = x_{ \rs{k}{j} }  - x_j\;,\;\; k   \in \{1,\dots,d-1\}. $$
  Let $x \in \cD$ and $j \in \{1, \dots, d\}$. For all $i \in \{1, \dots, d\}, i \neq j$, we have, by definition of $\cD$ and since $\sum_{k=1}^d P_{i,k} = 1$,
    \begin{align*}
      x_i - x_j \ge \sum_{k=1}^d P_{i,k} \left(x_k - x_j\right) - \bar{c}_i. 
    \end{align*}
    Thus $\pi_j(x)$ satisfies to
    \begin{align*}
      Q^{(j,j)} \pi^j(x) \succcurlyeq - \bar{c}^{(j)}.
    \end{align*}
    Since $\left(Q^{(j,j)}\right)^{-1} = \sum_{k \ge 0} \left(P^{(j,j)}\right)^k \succcurlyeq 0$, we obtain, using inequality \eqref{eq carac C}
    \begin{align} \label{def C}
      \pi^j(x) \succcurlyeq  - \left(Q^{(j,j)}\right)^{-1}\bar{c}^{(j)} = - C^{(j)}_{\cdot,j},
    \end{align}
    which means $x_i - x_j \ge - C_{i,j}$ for all $i \neq j$.\\
    Let $1 \le i \neq j \le d$. The precedent reasoning gives $x_i - x_j \ge - C_{i,j}$ and $x_j - x_i \ge - C_{j,i}$, thus \eqref{eq encadrement} is proved.\\
From  \eqref{eq encadrement}, we straightforwardly obtain  \eqref{eq cond C non empty} and the fact that $\Dslice$ is compact in $\{y \in \R^d : y_d = 0\}$.
\\ 1.b Since $\cD$ is non empty, the following holds for some $x \in \R^d$, recalling \eqref{eq de domain with Q},
\begin{align*}
Qx + \bar{c}  \succcurlyeq 0
\end{align*}
Multiplying by ${\mu}$ the previous inequality, we obtain \eqref{eq cond mean cost non empty}, since ${\mu}Q = 0$.
\\
2. Assume now that $\cD$ has non-empty interior and 
consider $x \in \overset{\circ}{\cD}$. Then, for all $1 \le i \le d$, we have that $x - \epsilon e_i$  belongs to $\cD$ for  $\epsilon > 0$ small enough. Thus, we get
\begin{align*}
x_i - \epsilon \ge \sum_{\ell = 1}^d P_{i,\ell} x_{\ell} - \epsilon P_{i,i} - \bar{c}_i
\end{align*}
and then
\begin{align*}
Qx+\bar{c} \succcurlyeq \epsilon \min_{1 \le i \le d}(1-P_{i,i}) \sum_{\ell = 1}^d e_\ell.
\end{align*}
Since $P$ is irreducible, $\min_{1 \le i \le d}(1-P_{i,i})  > 0$, and multiplying by ${\mu}$ both sides of the previous inequality we obtain ${\mu} \bar{c}>0$.
\\
For any $j \neq i$, since $x - \epsilon e_i \in \cD$, we deduce from \eqref{eq encadrement}, $-C_{i,j}+\epsilon \le x_i - x_j$. Using again \eqref{eq encadrement}, we get $-C_{i,j}+\epsilon \le C_{j,i}$. This proves the right hand side of \eqref{eq cond cost non empty interior}. \eproof
%

\vspace{2mm}
\noindent The next lemma, whose proof is postponed to  Appendix \ref{proof le costly round-trip}, links the condition \eqref{eq cond mean cost non empty} to costly round-trip.
\begin{Lemma}\label{le costly round-trip}
 The followings hold, for $1 \le j \le d$,
\begin{align} \label{eq barCjj linked to mu c}
\bar{C}_{jj} = \frac{\mu \bar{c}}{\mu_j} \;,
\end{align}
and, for $1 \le i \neq j \le d$,
\begin{align}\label{eq Cij+Cji linked to mu c}
C_{i,j} + C_{j,i} = \frac{\mu \bar{c}}{\mu_i}  \left(\left[Q^{(j,j)}\right]^{-1}\right)_{i^{(j)},i^{(j)}} \;.
\end{align}
\end{Lemma}

\vspace{2mm}

\noindent We are now going to show that previous necessary conditions are also sufficient. The main result of this section is thus the following.
\begin{Theorem} \label{th set of cns} The following conditions are equivalent:
  \begin{enumerate}[i)]
  \item The domain $\cD$ is non-empty (resp. has non-empty interior).
  \item There exists $1 \le i \neq j \le d$ such that $C_{i,j} + C_{j,i} \ge 0$ (resp. $C_{i,j} + C_{j,i} > 0$).
  \item The inequality $\mu \bar{c} \ge 0$ (resp. $\mu \bar{c} > 0$) is satisfied.
  \item For all $1 \le i \neq j \le d$, $C_{i,j} + C_{j,i} \ge 0$ (resp. $C_{i,j} + C_{j,i} > 0$).
  \end{enumerate}
\end{Theorem}

\proof 1 We first note that in Proposition \ref{pr necessary conditions} we have proved 
$i)\implies iv)$. We also remark that $iv) \implies ii)$ trivially, and $ii) \implies iii)$ in a straightforward way from equality \eqref{eq Cij+Cji linked to mu c}, recalling that $\left(Q^{(j,j)}\right)^{-1} = \sum_{k \ge 0} \left(P^{(j,j)}\right)^k \succcurlyeq 0$.
\\
2. We now study $iii) \implies i)$.
\\
2.a Assume that $\mu \bar{c} \ge 0$. For $1 \le j \le d$, we denote $z^j := -C_{\cdot,j}$. Then from \eqref{eq dpp C}, we straightforwardly observe that, for all $i \neq j$,
\begin{align}\label{eq calcul zji}
z^j_i = \sum_{\ell = 1}^d z^j_\ell P_{i,\ell} - \bar{c}_i \;. 
\end{align}
We now take care of the case $i=j$ by computing, recall $z^j_j=0$,
\begin{align}\label{eq calcul zjj}
z^j_j - \sum_{\ell = 1}^d z^j_\ell P_{j,\ell} + \bar{c}_j
&= \sum_{\ell = 1}^d C_{\ell,j} P_{j,\ell} + \bar{c}_j 
=\bar{C}_{j,j}
\end{align}
where we used \eqref{eq dpp bar C} with $i=j$. Then, combining  \eqref{eq barCjj linked to mu c} and the assumption $\mu \bar{c} \ge 0$ for this step, we obtain that $z^j \in \cD$ and so $\cD$ is non empty.
\\
2.b We assume that $\mu \bar{c}>0$ which implies that $\bar{C}_{jj} = \frac{\mu \bar{c}}{\mu_j} > 0$ for all $1 \leq j \leq d$ recalling \eqref{eq barCjj linked to mu c}. Set any $j \in \set{1,\dots,d}$ and consider $z^j := -C_{\cdot,j}$ introduced in the previous step.
We then set 
\begin{align}\label{eq de x}
x := z^j + \frac{1}{2(d-1)}\sum_{k \neq j}(z^k-z^j).
\end{align}
Next, we compute, for $i \neq j$, recalling \eqref{eq calcul zji} and \eqref{eq calcul zjj},
\begin{align*}
\left(Qx + \bar{c} \right)_i &= \left(Qz^j + \bar{c} \right)_i + \frac{1}{2(d-1)}\sum_{k \neq j}(Qz^k-Qz^j)_i\\
&= 0+\frac{1}{2d} (Qz^i-Qz^j)_i = \frac{1}{2(d-1)}((Qz^i)_i+\bar{c}_i) = \bar{C}_{i,i}>0.
\end{align*}
For $i=j$, we compute,  recalling \eqref{eq calcul zji} and \eqref{eq calcul zjj},
\begin{align*}
\left(Qx + \bar{c} \right)_j&= \left(Qz^j + \bar{c} \right)_j +  \frac{1}{2(d-1)}\sum_{k \neq j}(Qz^k-Qz^j)_j\\
&= \bar{C}_{j,j} +  \frac{1}{2(d-1)}\sum_{k \neq j} (-\bar{c}_j + \bar{c}_j -\bar{C}_{j,j}) = \frac{\bar{C}_{j,j}}{2}>0.
\end{align*}
Combining the two previous inequalities, we  obtain that
\begin{align*}
Qx + \bar{c} \succcurlyeq \frac{\delta}{2}  \1\, \quad \text{with} \quad \delta = \min\{\bar{C}_{i,i}|1 \leq i \leq d\}.
\end{align*}
From this, we easily deduce that $x + B(0, \frac{\delta}{4 \sup_i ||Q_{i,\cdot}||_2} ) \subset \cD$, which proves that $\cD$ has a non-empty interior.
\\
\eproof


\noindent We now give some extra conditions that are linked to the non-emptiness of the domain $\cD$ 

\begin{Proposition} The following assertions are equivalent:
\begin{enumerate}[i)]
\item $\cD$ is non-empty,
\item For all $1 \le i,j,k \le d$, the following holds
\begin{align}\label{eq nice cost}
C_{jk} \le C_{ji} + C_{ik},
\end{align}
\item \textcolor{black}{For any round trip of length less than $d$, i.e. $1\le n \le d$, $1 \le i_1 \neq...\neq i_n \le d$, we have
\begin{align}\label{eq any roundtrip}
{\sum_{k=1}^{n-1} C_{i_k,i_{k+1}} +  C_{i_n, i_1} \ge 0 }.
\end{align}}
\end{enumerate}
\end{Proposition}
\proof
1. $i) \implies ii)$. From Theorem \ref{th set of cns}, we know that $-C_{\cdot,k} \in \cD$ for all $k \in \set{1,\dots,d}$. Using then  \eqref{eq encadrement}, we have
\begin{align*}
-C_{i,k} + C_{j,k} \le C_{j,i}\;, \text{ for all}  1 \le i, j \le d\;, 
\end{align*}
which concludes the proof for this step.
\\
2. $ii) \implies iii)$ directly since $C_{i,i}=0$ for all $1 \le i \le d$. Finally $iii) \implies i)$ is already proved in Theorem \ref{th set of cns} for $2$-state round trip.
\eproof



\begin{Proposition} \label{pr convex hull}
Let us assume that $\cD$ has a non empty interior. Define  $\theta_{\cdot,j} = C_{\cdot,j} - C_{d,j} \1$, for all $1 \le j \le d$.
Then $(-\theta_{\cdot,j})_{1 \le j \le d}$ are affinely independent and  $\Dslice$ is the convex hull of these points.
\end{Proposition}

\proof We know from  Step 2.a in the proof of Theorem \ref{th set of cns} that $-C_{\cdot,j} \in \cD$ for all $1 \le j \le d$. The invariance by translation along $\1$ of the domain proves that $-\theta_{\cdot,j}$ are in $\Dslice$. More precisely, we obtain from \eqref{eq dpp C} that,
\begin{align}\label{eq property theta}
\theta_{i,j} - \sum_{\ell=1}^d \theta_{\ell,j}P_{i,\ell} = \bar{c}_i \;,\quad \text{ for } 1 \le i \neq j \le d\;.
\end{align}
\\
1. We now prove that $(\theta_{\cdot,j})_{1 \le j \le d}$ are affinely independent.
We consider thus $\alpha \in \R^d$ such that 
\begin{align}\label{eq ass on alpha}
\sum_{j=1}^d \alpha_j = 0 \quad 
\text{ and } \quad z := \sum_{j = 1}^d \alpha_j \theta_{\cdot,j} = 0\;.
\end{align}
and we aim to prove that $\alpha_j = 0$, for $j \in \set{1,\dots,d}$.
To this end, we compute, for $i \in \set{1,\dots,d}$,
\begin{align*}
z_i := \sum_{j=1}^d \alpha_j \theta_{i,j} &= \sum_{j \neq i} \alpha_j \theta_{i,j}  + \alpha_i \theta_{ii}
= \sum_{j\neq i} \alpha_j \bar{c}_i + \sum_{\ell = 1}^d \sum_{j \neq i}\alpha_j\theta_{\ell,j}P_{i,\ell}
-\alpha_i C_{d,i}
\\
&=\sum_{j\neq i} \alpha_j \bar{c}_j + \sum_{\ell = 1}^d z_\ell P_{i,\ell} - \alpha_i\sum_{\ell=1}^d \theta_{\ell,i} P_{i,\ell}
-\alpha_i C_{d,i}
\\
&= - \alpha_{i}(\bar{c}_i + \sum_{\ell=1}^d \theta_{\ell,i} P_{i,\ell} + C_{d,i})
= - \alpha_i(\bar{c}_i + \sum_{\ell=1}^d C_{\ell,i} P_{i,\ell})=-\alpha_i \bar{C}_{i,i}.
\end{align*}
We thus deduce that $\alpha_i=0$ since $\bar{C}_{i,i}=\frac{\mu \bar{c}}{\mu_i}>0$, which concludes the proof for this step.
\\
2.  We now show that $\Dslice$ is the convex hull of points $(-\theta_{\cdot,j})_{1 \le j \le d}$, which are affinely independent from the previous step. For $y \in \R^{d} \cap \{y_d = 0\}$, there exists thus a unique $(\lambda_1,\dots,\lambda_{d-1}) \in \R^{d-1}$ such that $y = \sum_{j=1}^d -\lambda_j \theta_{\cdot,j}$, with $\lambda_d = 1 - \sum_{j=1}^{d-1} \lambda_j$. Assuming that $y \in \cD$, we have that 
\begin{align*}
v := Qy + \bar{c} =  \sum_{j=1}^d -\lambda_j Q \theta_{\cdot,j}  + \bar{c} =\sum_{j=1}^d \lambda_j [Q(-\theta_{.,j})+\bar{c}]\succcurlyeq 0\;.
\end{align*}
Since $[Q(-\theta_{.,j})+\bar{c}]_i=0$ for all $i \neq j$, we get, for all $1 \le i \le d$,
$$v_i = \lambda_i ([Q(-\theta_{.,i})]_i+\bar{c}_i) \ge 0.$$
Recalling that $[Q(-\theta_{.,i})]_i+\bar{c}_i \ge 0$ we obtain $\lambda_i \ge 0$ which concludes the proof. \eproof

\subsubsection{The setting of controlled randomisation}
In this part we adapt Assumption \ref{ass-uncontrolled-irred} in the following natural way.
\begin{Assumption} \label{ass-uncontrolled-irred- new}
  For all $u \in \cC$, the Markov chain with stochastic matrix $P^u := (P^u_{i,j})_{1 \le i,j \le d}$ is irreducible. 
\end{Assumption}

\noindent We then consider the matrix $\widehat{C}$ defined by, for all $(i,j) \in \set{1,\dots,d}$
\begin{align}\label{eq de min cost}
 \widehat{C}_{i,j} := \min_{u \in \cC} {C}^u_{i,j} \, 
 \,  
\end{align}
recall the Definition of $\bar{C}^u_{i,j}$ for a fixed control in \eqref{eq de expected cost}. Let us note that $ \widehat{C}_{i,j}$ is well defined in $\R$ under Assumption \ref{ass-uncontrolled-irred- new} since $\cC$ is compact.

\noindent The following result is similar as Proposition \ref{pr necessary conditions}  but in the context of switching with controlled randomisation.\\

\begin{Proposition}\label{pr necessary conditions general} Assume $\cD$ is non-empty (resp. has non-empty interior). Then,
\begin{align}\label{eq cond cost non empty interior general}
\min_{u}\mu^u \bar{c}^u \ge 0\;(\text{resp.} > 0) 
\quad\text{ and } 
\quad\min_{1\le i\neq j \le d} \left(\widehat{C}_{i,j} + \widehat{C}_{j,i} \right) \ge 0\;(\text{resp.} > 0)\;.
\end{align}
Moreover, the set $\Dslice$ is compact in $\{y \in \R^d : y_d = 0\}$.
\end{Proposition}
\proof 1. Let $x \in \cD$.
From \eqref{eq encadrement}, we have for each $u \in \cC$,
$
-C^{u}_{i,j} \le x_i - x_j \le C^{u}_{j,i} \;.
$
Minimizing on $u \in \cC$, we then obtain
\begin{align}\label{eq encadrement general}
-\widehat{C}_{i,j}  \le x_i - x_j \le \widehat{C}_{j,i} \;.
\end{align}
From this, we deduce that $\Dslice$ is compact in $\{y \in \R^d : y_d = 0\}$ and we get the right handside of \eqref{eq cond cost non empty interior general}.
\\
We also have that, for all $u \in \cC$,
\begin{align*}
Q^u x + \bar{c}^u \succcurlyeq 0\,,
\end{align*}
then multiplying by $\mu^u$ we obtain $\mu^u \bar{c}^u \ge 0$. This leads to
$\min_{u}\mu^u \bar{c}^u \ge 0$.
\\
\textcolor{black}{2. Then, results concerning the non-empty interior framework can be obtained as in the proof of Proposition \ref{pr necessary conditions}.}

\paragraph{The case of controlled costs only.}
Let us first start by introducing the minimal  controlled mean cost:
\begin{align*}
\hat{c}_i := \min_{u \in \cC} \bar{c}^u_i\;,\quad \text{for}\; 1 \le i \le d\;.
\end{align*}
In this setting, we have that
\begin{align*}
  \cD 
   :=& \{x \in \R^d : (Qx)_i + \bar{c}^{u}_i \ge 0\;, \text{for all } u \in \cC, \; 1 \le i \le d\}
  \\
  = &\{x \in \R^d : (Qx)_i + \hat{c}_i \ge 0\;, \text{for all } \; 1 \le i \le d\}.
\end{align*}
Using the result of Proposition \ref{pr necessary conditions} with the new costs $\hat{c}$, we know that a necessary and sufficient condition is $\mu \hat{c} \ge 0$. Moreover, the matrix $C$ is defined here by
\begin{align}\label{eq carac C control cost}
C_{i,j} = \left((Q^{(j,j)})^{-1} \hat{c}^{(j)}\right)_{i-\1_{\set{i>j}} },\quad 1 \le i \neq j \le d \;,
\end{align}
and $C_{i,i} =0$, for $1 \le i \le d$. Comparing the above expression with the definition of $\widehat{C}$
in \eqref{eq de min cost}, we observe that $C_{i,j} \le \widehat{C}_{i,j}$, $1 \le i,j \le d$. The following example confirms that 
\begin{align*}
\min_{1\le i\neq j \le d} \left(\widehat{C}_{i,j} + \widehat{C}_{j,i} \right) \ge 0\,,
\end{align*}
recall Proposition \ref{pr necessary conditions general},  is not a sufficient condition in this context for non-emptiness of the domain.
\begin{Example}
Set $\cC = \set{0,1}$,
\begin{align*}
P = 
\left(
\begin{array}{ccc}
0 &  0.5  & 0.5 \\
0.5  & 0 & 0.5 \\
0.5 & 0.5 & 0 
\end{array}
\right)
\;,\;
\bar{c}^0 =\left(
\begin{array}{c}
-0.5  \\
1.2  \\
0.7 \\
\end{array}
\right)
\;
\text{ and }
\;
\bar{c}^1=\left(
\begin{array}{c}
1.5  \\
0.2   \\
0.2 \\
\end{array}
\right)
\end{align*}
Observe that $\mu = (\frac13,\frac13,\frac13)$ and $\hat{c} = (-0.5,0.2,0.2)^\top$.
Then, one computes that
\begin{align*}
\min_{1\le i\neq j \le d} \left(\widehat{C}_{i,j} + \widehat{C}_{j,i} \right) > 0
\quad \text{ but } \quad \mu \hat{c} < 0\;.
\end{align*}
\end{Example}


%
%
%
%

\subsection{The Markovian framework}
\label{sub se markov}

We now introduce a Markovian framework, and prove that a solution to \eqref{orbsde}-\eqref{orbsde2}-\eqref{orbsde3} exists for the randomised switching problem under Assumption \ref{ass-uncontrolled-irred} and a technical {copositivity} hypothesis, see Assumption \ref{ass-copositivity} below. \textcolor{black}{We also investigate an example of switching problem with controlled randomisation, see \eqref{transitions et couts exemple controle}.}\\
To this effect, we rely on the existence theorem obtained in \cite{chassagneux2018obliquely}, which we recall next.
\vspace{2mm}
\noindent For all $(t,x) \in [0,T] \times \R^q$, let $X^{t,x}$ be the solution to the following SDE:
\begin{align}
\label{eq SDE}
  \ud X_s &= b(s,X_s) \ud s + \sigma(s, X_s) \ud W_s, s \in [t,T], \\
  X_t &= x.
\end{align}
We are interested in the solutions $(Y^{t,x},Z^{t,x},K^{t,x}) \in \mathbb S^2_d(\F^0) \times \H^2_{d \times \kappa}(\F^0) \times \A^2_d(\F^0)$ of \eqref{orbsde}-\eqref{orbsde2}-\eqref{orbsde3}, where the terminal condition satisfies $\xi = g(X^{t,x}_T)$, and the driver satisfies $f(\omega,s,y,z) = \psi(s,X^{t,x}_s(\omega),y,z)$ for some deterministic measurable functions $g, \psi$. We next give the precise set of assumptions we need to obtain our results.


\noindent For sake of completeness, we recall here the existence result proved in \cite{chassagneux2018obliquely}, see also \cite{DAFH17}.
\begin{Assumption} \label{assumptionMarkov} 
 There exist $p \ge 0$ and $L \ge 0$ such that
  \begin{enumerate}[i)]
  \item
\begin{align*}
      |\psi(t,x,y,z)| \le L(1+|x|^p+|y|+|z|).
    \end{align*}
    Moreover, $\psi(t,x,\cdot,\cdot)$ is continuous on $\R^d\times\R^{d \times \kappa}$ for all $(t,x) \in [0,T]\times\R^q$.
  \item $(b,\sigma) : [0,T] \times \R^q \to \R^q \times \R^{q \times \kappa}$ is a measurable function satisfying, for all $(t,x,y) \in [0,T] \times \R^q \times \R^q$,
    \begin{align*}
      |b(t,x)| + |\sigma(t,x)| &\le L(1+|x|),\\
      |b(t,x)-b(t,y)|+|\sigma(t,x)-\sigma(t,y)|&\le L|x-y|.
    \end{align*}
  \item $g : \R^q \to \R^d$ is measurable and for all $(t,x) \in [0,T] \times \R^q$, we have
    \begin{align*}
      |g(t,x)| \le L(1+|x|^p).
    \end{align*}
  \item Let $\mathcal{X} = \{\mu(t,x;s,dy),x \in \R^q \textrm{ and } 0\leq t \leq s \leq T\}$ be the family of laws of $X^{t,x}$ on $\R^q$, i.e., the measures such that $\forall A \in \mathcal{B}(\R^q)$, $\mu(t,x;s,A) = \mathbb{P}(X_s^{t,x} \in A)$.
    For any $t \in [0,T)$, for any $\mu(0,a;t,dy)$-almost every $x \in \R^q$, and any $\delta \in ]0,T-t]$, there exists an application $\phi_{t,x}: [t,T]\times \R^d \rightarrow \R_+$ such that:
    \begin{enumerate}
    \item $\forall k \geq 1$, $\phi_{t,x} \in L^2([t+\delta,T] \times [-k,k]^q; \mu(0,a;s,dy)ds)$,
    \item $\mu(t,x;s,dy)ds = \phi_{t,x}(s,y)\mu(0,a;s,dy)ds$ on $[t+\delta,T] \times \R^q$.
    \end{enumerate}
  \item $H : \R^d \to \R^{d \times d}$ is a measurable function, and there exists $\eta > 0$ such that, for all $(y,y') \in \cD\times\R^d$ and $v \in \fn(\mathfrak{P}(y))$, where $\mathfrak{P}$ is the projection on $\cD$, we have
    \begin{align*}
      v^\top H(y) v &\ge \eta,\\
      |H(y')| &\le L.
    \end{align*}
    Moreover, $H$ is continuous on $\cD$.
  \end{enumerate}
\end{Assumption}

\begin{Remark} Assumption iv) is true as soon as $\sigma$ is uniformly elliptic, see \cite{HLP97}. \end{Remark}

\noindent The existence result in the Markovian setting reads as follows.

\begin{Theorem}[\cite{chassagneux2018obliquely}, Theorem 4.1] \label{thm-existence-cr}
  Under Assumption \ref{assumptionMarkov}, there exists a solution $(Y^{t,x},Z^{t,x},\Psi^{t,x}) \in \mathbb S^2_d(\F^0) \times \H^2_{d \times \kappa}(\F^0) \times \H^2_d(\F^0)$ of the following system
  \begin{align}
    \hspace{-1cm} \label{orbsde-gen-1} Y_s &= g(X^{t,x}_T) + \int_s^T \!\!\!\!\psi(u,X^{t,x}_u,Y_u,Z_u) \ud u - \int_s^T \!\!\!Z_u \ud W_u - \int_s^T \!\!\!\! H(Y_u) \Psi_u \ud u, s \in [t,T], \\
    \label{orbsde-gen-2} Y_s &\in \cD, \mbox{ } \Psi_s \in \cC(Y_s),\mbox{ } t \le s \le T,\\
    \label{orbsde-gen-3} \int_t^T &1_{\{Y_s \not \in \partial \cD\}} |\Psi_s| \ud s = 0.
  \end{align}
\end{Theorem}

The main point to invoke Theorem \ref{thm-existence-cr} is then to construct a function $H : \R^d \to \R^{d \times d}$ which satisfies  Assumption \ref{assumptionMarkov} v) and such that
\begin{align} \label{relation-cones}
  H(y) v \in \cC_o(y),
\end{align}
for all $y \in \cD$ and $v \in \cC(y)$, where $\cC_o(y)$ is the cone of directions of reflection, given here by
\begin{align*}
  \cC_o(y) := -\sum_{i=1}^d \R_+ e_i \ind{y_i = \max_{u \in \cC} \left\{ \sum_{j=1}^d P^u_{i,j} y_j - \bar{c}_{i}^u \right\}}.
\end{align*}
If Assumption \ref{assumptionMarkov} i), ii), iii), iv) are also satisfied, we obtain the existence of a solution to \eqref{orbsde-gen-1}-\eqref{orbsde-gen-2}-\eqref{orbsde-gen-3}. Setting $K^{t,x}_s := - \int_t^s H(Y^{t,x}_u) \Psi^{t,x}_u \ud u$ for $t \le s \le T$ shows that $(Y^{t,x},Z^{t,x},K^{t,x})$ is a solution to \eqref{orbsde}-\eqref{orbsde2}-\eqref{orbsde3}.

\subsubsection{Well-posedness result in the uncontrolled case }

We assume here Assumption \ref{ass-uncontrolled-irred}. In addition, we need to introduce the following technical assumption in order  to construct $H$ satisfying Assumption \ref{assumptionMarkov} v) and \eqref{relation-cones}.
\textcolor{black}{
\begin{Assumption} \label{ass-copositivity}
  For all $1 \le i \le d$, the matrix $Q^{(i,i)}$ is strictly copositive, meaning that for all $0 \preccurlyeq x \in \R^{d-1}, x \neq 0$, we have
  \begin{align}
    x^\top Q^{(i,i)} x > 0.
  \end{align}
\end{Assumption}
}
\noindent Our main result for this section is the following theorem.

\begin{Theorem} \label{constr H}
  Suppose that Assumption \ref{assumptionMarkov} i), ii), iii), iv), Assumption \ref{ass-uncontrolled-irred} and Assumption \ref{ass-copositivity} are satisfied and that $\cD$ has non-empty interior.
  
  \noindent Then, there exists $H : \R^d \to \R^{d \times d}$ satisfying  \ref{assumptionMarkov} v).
  Consequently, there exists a solution to \eqref{orbsde}-\eqref{orbsde2}-\eqref{orbsde3} with $\xi = g(X_T)$ and $f(\omega,s,y,z) = \psi(s,X^{t,x}_s(\omega),y,z)$. Moreover this solution is unique if we assume also Assumption \ref{hyp sup section 2}-ii). 
\end{Theorem}
\proof 
We first observe that uniqueness follows from Proposition \ref{prop unicite couts generaux}. We now focus on proving existence of solution which amounts to exhibit a convenient $H$ function. 
The general idea is to start by constructing $H$
on the points $(y^i)_{1 \le i \le d}$, given by
  \begin{align}
    y^i &:= (C_{d,i} - C_{j,i})_{1\le j \le d}, 
    \end{align}
then, using Proposition \ref{pr convex hull}, we can extend it on the whole $\Dslice$ by linear combination, and finally we extend $H$ on all $\R^d$ by using the geometry of $\cD$.
\\
The proof is then divided into several steps.
\vspace{1mm}
\\  
1. We start by computing the outward normal cone $\cC(y)$ for all $y \in \Dslice$. Let us set $y \in \Dslice$. Thanks to Proposition \ref{pr convex hull}, there exists a unique $(\lambda_i)_{1 \leqslant i \leqslant d} \in [0,1]^d$ such that 
$$y = \sum_{i=1}^d \lambda_i y^i, \quad \sum_{i=1}^d \lambda_i =1.$$
Let us denote $\mathcal{E}_y = \{1 \leqslant i \leqslant d |\lambda_i>0\}$.
We will show that 
\begin{align}
\label{resultat a prouver}
    \cC(y) = \sum_{j \notin \mathcal{E}_y} \R_+ n_j.
  \end{align}
  where 
   $
    n_i := (-Q_{i,j})_{1 \le j \le d},
  $
and with the convention $\cC(y)=\emptyset$ when $\mathcal{E}_y = \{1,...,d\}$.
Let us remark that the result is obvious when $\cC(y)=\emptyset$, since, in this case, $y$ is in the interior of $\cD$. So we will assume in the following that $\cC(y)\neq \emptyset$.\\
\noindent 1.a. First, let us show that for any $1 \leqslant i \leqslant d$, $(n_j)_{j \neq i}$ is a basis of $\{y \in \R^d | \sum_{k=1}^d v_k = 0\}$. Let $1 \le i \neq j \le d$. It is clear that $n_j \in \{v \in \R^d : \sum_{k=1}^d v_k = 0\}$. Since it is a hyperplan of $\R^d$ and that the family $(n_j)_{j \neq i}$ has $d-1$ elements, it is enough to show that the vectors are linearly independent. We observe that the matrix whose lines are the $n^{(i)}_j, j \neq i,$ is $- Q^{(i,i)}$. Since $P$ is irreducible, $Q^{(i,i)}$ is invertible. The vectors $n^{(i)}_j, j \neq i$ form a basis of $\R^{d-1}$, hence the vectors $(n_j)_{j \neq i}$ form a basis of $\{v \in \R^d | \sum_{k=1}^d v_k = 0\}$.\\
\noindent 1.b. We set now $j \notin \mathcal{E}_y$ and we will show that $n_j \in \cC(y)$. For any $z \in \cD$, by definition of $\cD$, we have
  \begin{align*}
    \bar{c}_j \ge \sum_{k=1}^d P_{j,k} z_k - z_j = n_j^\top z, 
    \end{align*}
and for all $i \in \mathcal{E}_y$, by definition of $y^{i}$, we have 
  \begin{align*}
    \bar{c}_{j} = \sum_{k=1}^d P_{j,k} y^{i}_k - y^{i}_j = n_j^\top y^{i}.
  \end{align*}
  This gives $n_j^\top (z - y) = n_j^\top z - \sum_{i \in \mathcal{E}_y} \lambda_{i}n_j^\top y^{i} \le 0$, hence $n_j \in \cC(y)$.\\
\noindent 1.c. We now set $i=\min \mathcal{E}_y$. Conversely, since $(n_j)_{j \neq i}$ is a basis of $\{v \in \R^d : \sum_{i=1}^d v_i = 0\} \ni \cC(y)$, see Lemma \ref{lem-general}, for $v \in \cC(y)$ there exists a unique $\alpha = (\alpha_j)_{j \neq i} \in \R^{d-1}$ such that $v = \sum_{j \neq i} \alpha_j n_j = (n_j)_{j \neq i} \alpha$. We will show here that $\alpha_{\ell}=0$ for all $\ell \in \mathcal{E}_y \setminus \{i\}$ and $\alpha_{\ell}\geq 0$ for all $\ell \notin \mathcal{E}_y$. \\
  For all $z \in \cD$, previous calculations give us:
  \begin{align*}
    0 &\ge \alpha^\top \left[(n_j)_{j \neq i}\right]^\top(z-y) = - \alpha^\top Q^{(i,\cdot)} (z-\sum_{\ell \in \mathcal{E}_y} \lambda_\ell y^\ell) = - \alpha^\top \left[Q^{(i,\cdot)}z - \sum_{\ell \in \mathcal{E}_y} \lambda_\ell Q^{(i,\cdot)}y^\ell \right].
  \end{align*}
  Let us recall that for any $j \neq i$, by definition of $y^j$, one gets $Q^{(i,\cdot)}y^j + \bar{c}^{(i)} = \frac{\mu \bar{c}}{\mu_j} e_j$, and $Q^{(i,\cdot)}y^i + \bar{c}^{(i)} = 0$.  Thus, previous inequality becomes
  \begin{align}\nonumber
    0 &\le  \alpha^\top \left[Q^{(i,\cdot)}z - \sum_{\ell \in \mathcal{E}_y\setminus \{i\} } \lambda_\ell \left(\frac{\mu \bar{c}}{\mu_\ell} e_\ell - \bar{c}^{(i)}\right) +\lambda_i \bar{c}^{(i)} \right]\\ \label{inegalite centrale}
    &= \alpha^\top \left[Q^{(i,\cdot)}z +\bar{c}^{(i)} - \sum_{\ell \in \mathcal{E}_y\setminus \{i\} } \lambda_\ell \frac{\mu \bar{c}}{\mu_\ell} e_\ell \right].
  \end{align}
  By taking $z=y^j$ in \eqref{inegalite centrale}, with $j \in \mathcal{E}_y\setminus \{i\}$, we get 
  \begin{align}
  \label{ineg alpha l}
   0&\le \alpha^\top \left[\frac{\mu \bar{c}}{\mu_j} e_j - \sum_{\ell \in \mathcal{E}_y\setminus \{i\} } \lambda_\ell \frac{\mu \bar{c}}{\mu_\ell} e_\ell \right]
  \end{align}
and so, we can sum, over $j$, previous inequality with positive weights $\alpha_j$, to obtain
  \begin{align*}
   0&\le \left(1-\sum_{j \in \mathcal{E}_y\setminus \{i\}} \lambda_j\right)\alpha^\top \left[\sum_{\ell \in \mathcal{E}_y\setminus \{i\} } \lambda_\ell \frac{\mu \bar{c}}{\mu_\ell} e_\ell \right].
  \end{align*}
  Then $0\le \alpha^\top \left[\sum_{\ell \in \mathcal{E}_y\setminus \{i\} } \lambda_\ell \frac{\mu \bar{c}}{\mu_\ell} e_\ell \right]$ since $\lambda_i>0$. Moreover, we have also $0\ge \alpha^\top \left[\sum_{\ell \in \mathcal{E}_y\setminus \{i\} } \lambda_\ell \frac{\mu \bar{c}}{\mu_\ell} e_\ell \right]$ by taking $z=y^i$ in \eqref{inegalite centrale}, which gives us that 
  \begin{align}
   \label{egalite sum alpha l}
   \alpha^\top \left[\sum_{\ell \in \mathcal{E}_y\setminus \{i\} } \lambda_\ell \frac{\mu \bar{c}}{\mu_\ell} e_\ell \right]=0.
  \end{align}
  We recall that $\mu \bar{c}>0$ since $\cD$ has non-empty interior (see Theorem \ref{th set of cns}). Pluging \eqref{egalite sum alpha l} in \eqref{ineg alpha l} gives us that $\alpha_j \ge 0$ for all $j \in \mathcal{E}_y\setminus \{i\}$, which, combined with \eqref{egalite sum alpha l} allows to conclude to $\alpha_j = 0$ for all $j \in \mathcal{E}_y\setminus \{i\}$.\\
  \noindent Now we apply \eqref{inegalite centrale} with $z=y^j$ for $j \notin \mathcal{E}_y$: hence $0 \leqslant \alpha_j \frac{\mu \bar{c}}{\mu_j}$ for all  $j \notin \mathcal{E}_y$, which concludes the proof of \eqref{resultat a prouver}.\\
\noindent 2. Then, we construct $H(y)$. Let us start by $H(y^i)$ for any $1 \le i \le d$. 
Fix $1 \le i \le d$, and let $B^i \in \R^{(d-1) \times (d-1)}$ be the base change matrix from $(-n_j^{(i)})_{j \neq i}$ to the canonical basis of $\R^{d-1}$. We set $H(y^i) := I^i B^i P^i$, with $I^i : \R^{d-1} \to \R^d$ and $P^i : \R^d \to \R^{d-1}$ the linear maps defined by
  \begin{align}
    I^i(x_1,\dots,x_{d-1}) &= (x_1,\dots,x_{i-1},0,x_i,\dots,x_{d-1}), \\
    P^i(x_1,\dots,x_d) &= (x_1,\dots,x_{i-1},x_{i+1},\dots,x_d).
  \end{align}
  Now we set $H(y) := \sum_{i \in\mathcal{E}_y} \lambda_i H(y^i)$. Let us take $v \in \cC(y)$. Thanks to \eqref{resultat a prouver}, we know that $v = \sum_{j=1}^d \alpha_j n_j$ for some $(\alpha_j)_{1 \le j \le d} \in (\mathbb{R}^+)^d$ and such that $\alpha_j=0$ when $j \in \mathcal{E}_y$. Since $n_k = - Q_k^\top$, for all $1\le k \le d$, we have $v = -Q^\top \alpha$.
  By construction, we get that 
  $$H(y)v = - \sum_{j \notin \mathcal{E}_y} \alpha_j e_j= -\alpha \in \cC_o(y).$$
  It remains to check that Assumption \ref{assumptionMarkov}-v) is fulfilled. If $v \neq 0$, which is equivalent to  $\alpha \neq 0$, we have, for $i \in \mathcal{E}_y$,
  \begin{align*}
   v^\top H(y)v= \alpha^\top Q \alpha= (\alpha^{(i)})^\top Q^{(i,i)} \alpha^{(i)} >0, 
  \end{align*}
  due to Assumption \ref{ass-copositivity} and the fact that $\alpha_i=0$.\\
\noindent 3. We have constructed $H$ on $\Dslice$ with needed properties. 
  Finally, we set $H(x) = H(x - x_d \sum_{i=1}^d e_i)$ for all $x \in \cD$ and $H(x) = H(\mathfrak{P}(x))$ for $x \in \R^d$ and the proof is finished.
  \eproof

\begin{Remark} i) Assumption \ref{ass-copositivity} is satisfied as soon as $P$ is symmetric and irreducible. Indeed,  $Q^{(i,i)}$ is then nonsingular, symmetric and diagonally dominant, hence positive definite, for all $i \in \set{1,\dots,d}$.
\\
ii) In dimension $3$, if $P$ is irreducible, then Assumption \ref{ass-copositivity} is automatically satisfied. Indeed, we have
  \begin{align}
    P = \left(\begin{array}{ccc}
                0 & p & 1-p \\
                q & 0 & 1-q \\
                r & 1-r & 0
              \end{array}\right)
  \end{align}
  for some $p,q,r \in [0,1]$ satisfying to $0 \le p+q, 1+r-p, 2-(q+r) < 2$ by irreducibility. Thus, for $i=1$ for example,
  \begin{align}
    Q^{(1,1)} + \left(Q^{(1,1)}\right)^\top = \left(\begin{array}{cc}
                                                      2 & -(p+q) \\
                                                      -(p+q) & 2
                                                    \end{array}\right)
  \end{align}
  is nonsingular, symmetric and diagonally dominant, hence positive definite. Thus $x^\top Q^{(1,1)} x = \frac{1}{2} x^\top \left(Q^{(1,1)} + \left(Q^{(1,1)}\right)^\top\right) x > 0$ for all $x \neq 0$.
\\
iii)  However, in dimension greater than $3$, it is not always possible to construct a function $H$ satisfying to Assumption \ref{assumptionMarkov}. For example in dimension $4$, consider the following matrix:
  \begin{align}
    P = \left(\begin{array}{cccc}
                0 & \frac{\sqrt{3}}{2} & 0 & 1 - \frac{\sqrt{3}}{2} \\
                1 - \frac{\sqrt{3}}{2} & 0 & \sqrt{3} - 1 & 1 - \frac{\sqrt{3}}{2} \\
                0 & 1 & 0 & 0 \\
                \frac13 & \frac13 & \frac13 & 0
              \end{array}\right),
  \end{align}
  together with positive costs $c$ to ensure that the domain has non-empty interior.\\
  It is an irreducible stochastic matrix, and let's consider the extremal point $y^4$ such that
  \begin{align}
    y^4_4 &= 0, \\
    y^4_1 &= \frac{\sqrt 3}{2} y^4_2 - c_1, \\
    y^4_2 &= (1 - \frac{\sqrt 3}{2})y^4_1 + (\sqrt 3 - 1)y^4_3 - c_2, \\
    y^4_3 &= y^4_2 - c_3.
  \end{align}
  We have $\cC(y^4) = \R_+ (-1, \frac{\sqrt 3}{2}, 0, 1 - \frac{\sqrt 3}{2})^\top + \R_+ (1 - \frac{\sqrt 3}{2}, -1, \sqrt 3 - 1, 1 - \frac{\sqrt 3}{2})^\top + \R_+(0,1,-1,0)^\top =: \sum_{i=1}^3 \R_+ n_i$.\\
  If $H(y^4)$ satisfies $H(y^4) n_1 = (-1,0,0,0), H(y^4) n_2 = (0,-1,0,0)$ and $H(y^4) n_3 = (0,0,-1,0)$, consider $v = \frac12 n_1 + n_2 + \frac{\sqrt 3}{2}n_3 \in \cC(y^4)$. Then it is easy to compute $v^\top H v = 0$, hence it is not possible to construct $H(y^4)$ at this point satisfying Assumption \ref{assumptionMarkov}.
\end{Remark}

\subsubsection{An example of switching problem with controlled randomization}

We assume here that $\cC = [0,1]$ and we consider the example of switching problem with controlled randomisation given by \eqref{transitions et couts exemple controle}. 
Since the cost functions are positive, $\cD$ has a non-empty interior.
\begin{Theorem}
\label{th existence  markov ex controled}
There exists a function $H :\R^3\to \R^{3 \times 3}$ that satisfies Assumption \ref{assumptionMarkov}-v) and such that
\begin{align*} 
  H(y) v \in \cC_o(y), \quad \forall y \in \cD,\, v \in \cC(y). 
\end{align*}
 Consequently, if we assume that Assumption \ref{assumptionMarkov}(i)-(iv) is fulfilled, there exists a solution to \eqref{orbsde}-\eqref{orbsde2}-\eqref{orbsde3} with $\xi = g(X_T)$ and $f(\omega,s,y,z) = \psi(s,X^{t,x}_s(\omega),y,z)$. Moreover this solution is unique if we assume also Assumption \ref{hyp sup section 2}-ii).
 %
\end{Theorem}



\begin{proof}
We first observe that uniqueness follows once again from Proposition \ref{prop unicite couts generaux}.\\
\noindent 1. We start by constructing $H$ on the boundary of $\cD$. Recalling Lemma \ref{lem-general}, it is enough to construct it on its intersection with $\Dslice$ which is
 made up of $3$ vertices
 $$y^1=(1,0,0), \quad y^2=(0,1,0), \quad y^3=(0,-1,-1)$$
 and three edges that are smooth curves. We denote $\mathcal{E}_1$ (respectively $\mathcal{E}_2$ and $\mathcal{E}_3$) the curve between $y^1$ and $y^2$ (respectively between $y^2$ and $y^3$ and between $y^3$ and $y^1$). Let us construct $H(y^1)$ and $H(y^2)$: we must have
 $$H(y^1) \left(\begin{array}{cc}
1 & 1\\
0 & -1\\
-1 & 0
\end{array} \right) = \left(\begin{array}{cc}
0 & 0\\
0 & -b\\
-a & 0
\end{array} \right), \quad 
 H(y^2)\left(\begin{array}{cc}
-1 & 0\\
1 & 1\\
0 & -1
\end{array} \right) = \left(\begin{array}{cc}
-c & 0\\
0 & 0\\
0 & -d
\end{array} \right),$$
 with $a,b,c,d>0$. Let us set $a=b=c=d=1$. Then we can take 
 $$H(y^1)=\left(\begin{array}{ccc}
1 & 1 & 1\\
1 & 2 & 1\\
1 & 1 & 2
\end{array} \right), \quad  H(y^2)=\left(\begin{array}{ccc}
2 & 1 & 1\\
1 & 1 & 1\\
1 & 1 & 2
\end{array} \right).$$
We define now $H$ on $\mathcal{E}_1$. We denote $(x_s)_{s \in [0,1]}$ a continuous parametrization of $\mathcal{E}_1$ such that $x_0=y^1$ and $x_1=y^2$. For all $s \in [0,1]$, we also denote $R_s$ the matrix that send the standard basis on a local basis at point $x_s$ with the standard orientation and such that: the two first vectors are in the plane $\{z=0\}$, the first one is orthogonal to $\mathcal{E}_1$ while the second one is tangent to $\mathcal{E}_1$ and the third one is $e_3$. We have in particular, $Q_0=\text{Id}$. Then we just have to set
$$H(x_s)=R_s [s H(y^1) +(1-s) R_1^{-1} H(y^2) R_1]R_s^{-1}.$$
We can check that, by construction, Assumption \ref{assumptionMarkov}-v) and \eqref{relation-cones} are fulfilled for points on $\mathcal{E}_1$.
Moreover, we are able to construct by the same method $H$ on $y^3$, and then on $\mathcal{E}_2$ and $\mathcal{E}_3$, satisfying Assumption \ref{assumptionMarkov}-v) and \eqref{relation-cones}. \\
\noindent 2. By using Lemma \ref{lem-general} we can extend $H$ on all the boundary of $\cD$. Finally, we can extend $H$ by continuity on the whole space $\mathbb{R}^3$ by following Remark 2.1 in \cite{chassagneux2018obliquely}.
 \eproof
\end{proof}
\subsection{The non-Markovian framework}
\label{sub se non markov}
We now switch to the non-Markovian case, which is more challenging. We prove the well-posedness of the RBSDE in the uncontrolled setting for two cases: Problems in dimension $3$ and the example of a symmetric transition matrix $P$, in any dimension.
\\
We first recall  Proposition 3.1 in \cite{chassagneux2018obliquely} that gives an existence result for non-Markovian obliquely reflected BSDEs and the corresponding assumptions, see  Assumption \ref{assumptionNonMarkov} below. Let us remark that the non-Markovian case is more challenging for our approach as it requires more structure condition on $H$, which must be symmetric and smooth in this case.

\begin{Assumption} \label{assumptionNonMarkov}There exists $L>0$ such that
  \begin{enumerate}[i)]
   \item $\xi:=g((X_t)_{t \in [0,T]})$ with $g : C([0,T],\mathbb{R}^q) \rightarrow \bar{\cD}$ a bounded uniformly continuous function and $X$ solution of the SDE \eqref{eq SDE} where $(b,\sigma) : [0,T] \times \R^q \to \R^q \times \R^{q \times \kappa}$ is a measurable function satisfying, for all $(t,x,y) \in [0,T] \times \R^q \times \R^q$,
    \begin{align*}
      |\sigma(t,x)| &\le L,\\
      |b(t,x)-b(t,y)|+|\sigma(t,x)-\sigma(t,y)|&\le L|x-y|.
      \end{align*}
   \item $f: \Omega \times [0,T] \times \mathbb{R}^d \times \mathbb{R}^{d \times \kappa} \rightarrow \mathbb{R}^d$ is a $\mathcal{P} \otimes \mathcal{B}(\mathbb{R}^d \times \mathbb{R}^{d \times \kappa})$-measurable function such that,  for all $(t,y,y',z,z') \in [0,T] \times \mathbb{R}^d \times \mathbb{R}^d \times \mathbb{R}^{d \times \kappa}\times \mathbb{R}^{d \times \kappa}$,
 \begin{align*} 
 |f(t,y,z) - f(t,y',z')| &\leqslant L \left(|y-y'| + |z-z'|\right)\;.
 \end{align*}
 Moreover we have 
 \begin{align*}
   \esssup_{\omega \in \Omega, t \in [0,T]} \mathbb{E}\left[ \int_t^T |f(s,0,0)|^2 \ud s \Big| \mathcal{F}_t\right] \leqslant L.
 \end{align*}
  \item $H:  \R^d \rightarrow \mathbb{R}^{d\times d}$ is valued in the set of symmetric matrices $Q$ satisfying 
 \begin{align}\label{eq bound matrix H}
 |Q| \le L\;,\quad L |\upsilon|^2 \ge  \upsilon^\top Q \upsilon \ge \frac1L |\upsilon|^2\,,\;\forall \upsilon \in \R^d.
\end{align}
 $ H$ is a $\cC^{1}$-function  and $H^{-1}$ is a  $\cC^{2}$ function satisfying
 \begin{align*}
 |\partial_y H| + |H^{-1}| +  |\partial_y H^{-1}| + |\partial^2_{yy} H^{-1}| \le L.
 \end{align*}
 \end{enumerate}
\end{Assumption}

\noindent From this assumption, follows the following general existence result in the non-Markovian setting.

\begin{Theorem}[\cite{chassagneux2018obliquely}, Proposition 3.1] \label{thm-existence-nonmarkov}
  We assume that $\cD$ has non-empty interior. Under Assumption \ref{assumptionNonMarkov}, there exists a solution $(Y,Z,\Psi) \in \mathbb S^2_d(\F^0) \times \H^2_{d \times \kappa}(\F^0) \times \H^2_d(\F^0)$ of the following system
  \begin{align}
    \hspace{-1cm} \label{orbsde-gen-nm-1} Y_s &= \xi + \int_s^T f(u,Y_u,Z_u) \ud u - \int_s^T Z_u \ud W_u - \int_s^T H(Y_u) \Psi_u \ud u, s \in [0,T], \\
    \label{orbsde-gen-nm-2} Y_s &\in \cD, \mbox{ } \Psi_s \in \cC(Y_s),\mbox{ } 0 \le s \le T,\\
    \label{orbsde-gen-nm-3} \int_0^T &1_{\{Y_s \not \in \partial \cD\}} |\Psi_s| \ud s = 0.
  \end{align}
\end{Theorem}

\begin{Remark}
\begin{enumerate}[i)]
 \item \textcolor{black}{The assumption on the terminal condition is slightly less general than the one needed in \cite{chassagneux2018obliquely} (see Assumption \textbf{SB}(i) and Corollary 2.2 in \cite{chassagneux2018obliquely}). One could get a more general result by assuming that $\mathbb{E}[\xi|\mathcal{F}_.]$ is a BMO martingale such that its bracket has sufficiently large exponential moment.}
 \item  We do not use Theorem 3.1 in \cite{chassagneux2018obliquely} since the domain $\cD$ is not smooth enough to apply it (see Assumption \textbf{SB}(iv) in \cite{chassagneux2018obliquely}). Consequently, we have to assume the extra assumption that $\xi$ is bounded. 
 \item The uniqueness result for this part is obtain also by invoking  Corollary \ref{co uniqueness RBSDE}.
\end{enumerate}
\end{Remark}

\subsubsection{Existence of solutions in dimension $3$}

We focus in this part on the uncontrolled case $\mathscr C = \{0\}$, in dimension $d=3$. Thus, there is a unique transition matrix given by
\begin{align}
  P:=P^0 =\left(\begin{array}{ccc}
          0&p&1-p\\
          q&0&1-q\\
          r&1-r&0\\
        \end{array}\right),
\end{align}
for some $p,q,r \in [0,1]$.

\begin{Theorem}
\label{th existence non markov dim3}
 Let us assume that $0<p,q,r<1$ and that $\cD$ has non-empty interior. Then there exists a function $H :\R^3\to \R^{3 \times 3}$ that satisfies Assumption \ref{assumptionNonMarkov}(iii) and such that
\begin{align} \label{relation-cones-2}
  H(y) v \in \cC_o(y), \quad \forall y \in \cD,\, v \in \cC(y). 
\end{align}
 Consequently, if we assume that Assumption \ref{assumptionNonMarkov}(i)-(ii) is fulfilled, then there exists a solution to the Obliquely Reflected BSDE \eqref{orbsde}-\eqref{orbsde2}-\eqref{orbsde3}.  Moreover this solution is unique if we assume also Assumption \ref{hyp sup section 2}-ii).
\end{Theorem}

\begin{proof}
 Once again we exhibit a convenient $H$. Thanks to Lemma \ref{lem-general}, it is enough to construct $H$ only on $\mathbb{R}^3 \cap \{(x,y,z) \in \mathbb{R}^3| z = 0\}$. we start by $\Dslice$ which is
 a triangle with three vertices $v^i = (x_i,y_i,z_i), i = 1,2,3$ given by:
\begin{align}
  x_1 &= p y_1 + (1-p) z_1 - c_1, x_2 = p y_2 + (1-p) z_2 - c_1, y_3 = q x_3 + (1-q) z_3 - c_2, \\
  y_1 &= q x_1 + (1-q) z_1 - c_2, z_2 = r x_2 + (1-r) y_2 - c_3, z_3 = r x_3 + (1-r) y_3 - c_3,\\
  z_1 &=0, z_2=0, z_3=0.
\end{align}
We first observe that uniqueness follows once again from Proposition \ref{prop unicite couts generaux}.
Let us now construct $H$ on each vertex. We consider first the point $v^1$. It is easy to compute its outward normal cone, which is given by
\begin{align}
  \cC(v^1) = \R^+ (-1, p, 1-p)^\top + \R^+ (q, -1, 1-q)^\top.
\end{align}
The matrix $H(v^1)$ must satisfy
\begin{align} \label{identity}
  H(v^1)\left(\begin{array}{cc}
           -1 & q \\
           p & -1 \\
           1-p & 1-q \\
         \end{array}\right) = \left(\begin{array}{cc}
                                     -a & 0 \\
                                     0 & -b \\
                                     0 & 0
                                    \end{array}\right)
\end{align}
for some $a,b > 0$.
Taking $a=\frac 1 q, b= \frac 1 p$, we consider,
for any $\alpha > 0$, 
\begin{align}
  H(v^1) &= -\left(\begin{array}{cc}
                   1 & 0 \\
                   0 & 1 \\
                   0 & 0
                 \end{array}\right)
                       \left(\begin{array}{cc}
                               -q & pq \\
                               pq & -p
                             \end{array}\right)^{-1}
                                   \left(\begin{array}{ccc}
                                           1 & 0 & 0 \\
                                           0 & 1 & 0 \\
                                           0 & 0 & 0
                                         \end{array}\right) + \left(\begin{array}{ccc}
                                                                      \alpha & \alpha & \alpha \\
                                                                      \alpha & \alpha & \alpha \\
                                                                      \alpha & \alpha & \alpha
                                                                    \end{array}\right)
                                                                      \\
         &= \frac{1}{pq(1-pq)}\left(\begin{array}{ccc}
                                      \alpha + p & \alpha + pq & \alpha\\
                                      \alpha + pq & \alpha + q & \alpha \\
                                      \alpha & \alpha & \alpha
                                    \end{array}\right).
\end{align}
It is easy to check that this matrix $H(v^1)$ is symmetric and positive definite for any $\alpha > 0$, so we can set $\alpha=1$ in the following. 
Similarly, we construct $H$ on vertices $v^2, v^3$,
\begin{align}
  H(v^2) &= \frac{1}{r(1-p)(1-r(1-p))}\left(\begin{array}{ccc}
                                      1 + (1-p) & 1 & 1 + r(1-p)\\
                                      1 & 1 & 1 \\
                                      1 + r(1-p) & 1 & 1 + r
                                         \end{array}\right), \\
  \nonumber H(v^3) &= \frac{1}{(1-q)(1-r)(1-(1-q)(1-r))}\left(\begin{array}{ccc}
                                      1 & 1 & 1 \\
                                      1 & 1 + (1-q) & 1 + (1-q)(1-r) \\
                                      1 & 1 + (1-q)(1-r) & 1 + (1-r)
                                         \end{array}\right).
\end{align}
We can extend $H$ on all $\Dslice$ by convex combination, i.e. linear interpolation. By this way, $H$ stays valued in the set of positive definite symmetric matrices and is smooth enough. We could also define $H$ outside $\cD \cap \{(x,y,z) \in \mathbb{R}^3| z = 0\}$ by linear interpolation but we will lose the boundedness and the positivity of $H$. Nevertheless we can find a bounded and convex, $\mathcal{C}^2$ open neighborhood $\mathcal{V}$ of $\cD$, small enough, such that $H$ (still defined by linear interpolation) stays valued in the set of positive definite symmetric matrices on $\overline{\mathcal{V}}$. Then we define $H(y)$ for $y \notin \overline{\mathcal{V}}$ by $H(\mathcal{P}(y))$ where $\mathcal{P}$ stands for the projection onto $\overline{\mathcal{V}}$. By this way, $H$ is a bounded function with values in the set of positive definite symmetric matrices, that satisfies \eqref{eq bound matrix H}, \eqref{relation-cones-2} and that is $\mathcal{C}^{0}(\mathbb{R}^2) \cap \mathcal{C}^{2}
(\mathbb{R}^2 \setminus \partial \mathcal{V})$ smooth, with $\partial \mathcal{V}$ the boundary of $\mathcal{V}$. Finally, we just have to mollify $H$ in a neighborhood of $\partial \mathcal{V}$, small enough to stay outside $\cD \cap \{z = 0\}$.
 \eproof
\end{proof}

\begin{Remark}
 When $pqr(1-p)(1-q)(1-r)=0$ then we can show that it is not possible to construct a function $H$ that satisfies Assumption \ref{assumptionNonMarkov}(iii) and \eqref{relation-cones-2}.
\end{Remark}

\subsubsection{Existence of solutions for a symmetric multidimensional example}
\label{section exemple symetrique}

We focus in this part on the uncontrolled case $\mathscr C = \{0\}$, in dimension $d \geqslant 3$ with a unique transition matrix $P$ given by
\begin{equation*}
P_{i,j} = \frac{1}{d-1} \1_{i \neq j}.
\end{equation*}

\begin{Theorem}
\label{th existence non markov example}
Assume that $\cD$ has non-empty interior. There exists a function $H :\R^d\to \R^{d \times d}$ that satisfies Assumption \ref{assumptionNonMarkov}(iii) and such that
\begin{align*} 
  H(y) v \in \cC_o(y), \quad \forall y \in \cD,\, v \in \cC(y). 
\end{align*}
 Consequently, if we assume that Assumption \ref{assumptionNonMarkov}(i)-(ii) is fulfilled, then there exists a solution to the Obliquely Reflected BSDE \eqref{orbsde}-\eqref{orbsde2}-\eqref{orbsde3}.  Moreover this solution is unique if we assume also Assumption \ref{hyp sup section 2}-ii).
\end{Theorem}
\begin{proof}
The proof follows exactly the same lines as the proof of Theorem \ref{th existence non markov dim3}.
 $\Dslice$ is a convex polytope with $d$ vertices $(y^i)_{1 \leqslant i \leqslant d}$ satisfying: for all $1 \leqslant i \leqslant d$,
 \begin{align*}
  y^i_{\ell} = \sum_{j \neq i }\frac{1}{d-1} y^j_{\ell} - \bar{c}_i, \quad \forall i \neq \ell, \quad \text{and} \quad y^i_d=0. 
 \end{align*}
Let us construct $H$ on vertex $y^d$. It is easy to compute its outward normal cone, which is positively generated by vectors $f^1,...,f^{d-1}$ where
$$f^k_i = -\1_{i=k} +\frac{1}{d-1} \1_{i \neq k}.$$
For any $1 \leqslant k \leqslant d-1$, we impose $H(y^d)f^k = -\alpha_k e_k$ with $\alpha_k>0$. We can check that it is true with $\alpha_k=1$ for all $1 \leqslant k \leqslant d-1$, if we set, for any $a>0$,
\begin{equation*}
H(y^d)=\left(\begin{array}{cccc}
a &  & a-\frac{d-1}{d} & a-2\frac{d-1}{d}\\
 & \ddots &  & \vdots\\
a-\frac{d-1}{d} &  & a & \vdots\\
a-2\frac{d-1}{d} & \ldots & \ldots & a-2\frac{d-1}{d}
\end{array}\right) .
\end{equation*}
Since $\frac{d-1}{d}$ is an eigenvalue of $H(y^d)$ with multiplicity $d-2$, $\Det(H(y^d))=\left(a-2\frac{d-1}{d}\right) (d-1) \left(\frac{d-1}{d}\right)^{d-2}$ and $\Trace(H(y^d)) =da-2\frac{d-1}{d}$, $H(y^d)$ is a positive definite symmetric matrix as soon as $a>2\frac{d-1}{d}$. Thus we can set $a=2$. By simple permutations of rows and columns in $H(y^d)$ we can construct easily $H(y^k)$ for any $1 \leqslant k \leqslant d$. Then we just have to follow the proof of Theorem \ref{th existence non markov dim3} to extend $H$ from vertices of $\Dslice$ to the whole space. 
 \eproof
\end{proof}


\appendix
\section{Appendix}

\subsection{Proof of Lemma \ref{le costly round-trip}}\label{proof le costly round-trip} 

\noindent For all $\cI, \cJ \subset \{1,\dots,d\}^2$, let $\mathrm{ad}[Q^{(\cI,\cJ)}]$  be the adjunct matrix of $Q^{(\cI,\cJ)}$. 
\\
For $1\le j \le d$, we denote, for ease of presentation, $\mathfrak{Q}^j:=\mathrm{ad}[Q^{(j,j)}] $, and we have
\begin{align} \label{eq de adjunct}
\mathfrak{Q}^j_{\ls{i}{j} ,\ls{\ell}{j} } = (-1)^{\ls{i}{j} + \ls{\ell}{j}} \det Q^{(\{j,\ell\},\{j,i\} )}.
\end{align}
for all $(\ell,i) \in \{1,\dots,d\} \backslash \{j\}$.
For all $1 \le i \neq j \le d$, we define 
\begin{align*}
C_{i,j} := \left((Q^{(j,j)})^{-1} \bar{c}^{(j)}\right)_{i-\1_{\set{i>j}} }  \text{ and }  C_{j,j} = 0. 
\end{align*}
Using $\mathfrak{Q}^j$ the adjunct matrix of $Q^{(j,j)}$, we observe then, for latter use,
\begin{align}\label{eq expression Cij}
C_{i,j} = \frac1{\mu_j} \sum_{\ell\neq j} \mathfrak{Q}^j_{\ls{i}{j},\ls{\ell}{j}}\bar{c}_{\ell}\;, \text{ for } i \neq j\;.
\end{align}

\proof
1. We first show that \eqref{le costly round-trip} holds true.
From \eqref{eq de expected cost}, we observe that
\begin{align*}
\bar{C}_{j,j} &= \esp{\sum_{n=0}^{\tau_j - 1} \bar{c}_{X_n} | X_0 = j}
\\
&= \esp{\sum_{n=0}^{\tau_j - 1} \sum_{\ell =1}^d \bar{c}_{\ell} \1_{\set{X_n=\ell}} | X_0 = j}.
\end{align*}
Thus,
\begin{align*}
\bar{C}_{j,j} = \sum_{\ell=1}^d  \bar{c}_{\ell} \gamma^j_\ell \; \text{ with }  \; \gamma^j_\ell  =\esp{\sum_{n=0}^{\tau_j-1} \1_{\set{X_n=\ell}} | X_0 = j}
\end{align*}
From \cite[Theorem 1.7.5]{norris1998markov}, we know that $\gamma^j_\ell = \frac{\mu_\ell}{\mu_j}$.
\\
\\
\vspace{2mm}
2. We prove \eqref{eq Cij+Cji linked to mu c} assuming the following for the moment: for all distinct $1 \le i,j,k \le d$,
\begin{align}\label{eq kind of magic}
\mu_i  \mathfrak{Q}^j_{i^{(j)}, k^{(j)}} + \mu_j \mathfrak{Q}^i_{j^{(i)}, k^{(i)}}  = \mathfrak{Q}^i_{j^{(i)}, j^{(i)}}\mu_k =  \mathfrak{Q}^j_{i^{(j)}, i^{(j)}} \mu_k.
\end{align}
Let $1 \le i \neq j \le d$. We have, using \eqref{eq expression Cij},
    \begin{align*}
      \nonumber C^{i,j} + C^{j,i} 
      \nonumber  &= \frac{1}{\mu_j} \left(\mathfrak{Q}^j \bar{c}^{(j)}\right)_{i^{(j)}} + \frac{1}{\mu_i} \left(\mathfrak{Q}^i \bar{c}^{(i)}\right)_{j^{(i)}} \\
      \nonumber &= \frac{1}{\mu_j} \sum_{k \neq j} \mathfrak{Q}^j_{i^{(j)},k^{(j)}} \bar{c}_k + \frac{1}{\mu_i} \sum_{k \neq i} \mathfrak{Q}^i_{j^{(i)},k^{(i)}} \bar{c}_k \\
      &= \frac{1}{\mu_j} \mathfrak{Q}^j_{i^{(j)},i^{(j)}} \bar{c}_i + \frac{1}{\mu_i} \mathfrak{Q}^i_{j^{(i)},j^{(i)}} \bar{c}_j + \sum_{k \neq i,j} \frac{\mu_i \mathfrak{Q}^j_{i^{(j)},k^{(j)}} + \mu_j \mathfrak{Q}^i_{j^{(i)},k^{(i)}}}{\mu_i \mu_j} \bar{c}_k.
    \end{align*}
    Using the previous point and the fact that $\mathfrak{Q}^j_{i^{(j)},i^{(j)}} = \mathfrak{Q}^i_{j^{(i)},j^{(i)}}$, we get
    \begin{align}
      \nonumber C^{i,j} + C^{j,i} &= \mathfrak{Q}^j_{i^{(j)},i^{(j)}} \left( \frac{\mu_i \bar{c}_i + \mu_j \bar{c}_j}{\mu_i \mu_j} + \sum_{k \neq i,j} \frac{\mu_k \bar{c}_k}{\mu_i \mu_j} \right) = \frac{\mathfrak{Q}^j_{i^{(j)},i^{(j)}}}{\mu_i \mu_j} \mu \bar{c}
    \end{align}
which is the result we wanted to prove.\\
3. We now prove \eqref{eq kind of magic}.\\
 Let $i,j \in \set{1,\dots, d}$ and $i \neq j$.
We observe first, using \eqref{eq de adjunct}, that
  \begin{align*}
  \mathfrak{Q}^j_{\ls{i}{j} ,\ls{i}{j} } = \det Q^{(\{j,i\},\{j,i\} )} =   \mathfrak{Q}^i_{\ls{j}{i} ,\ls{j}{i} } 
  \end{align*}
 For $k \in \set{1,\dots,d}\setminus \set{i,j}$, we denote by $k_{ij} \in \set{1,\dots,d-2}$ (resp. $i_{jk}$, $j_{ik}$) the index such that:
  \begin{align}
  Q_{k,\cdot} = Q^{(\{j,i\},\{j,i\} )}_{k_{ij}, \cdot}
  \,(\text{resp. }
   Q_{i,\cdot} = Q^{(\{j,k\},\{j,i\} )}_{i_{jk}, \cdot}\,,
    Q_{j,\cdot} = Q^{(\{k,i\},\{j,i\} )}_{j_{ik}, \cdot}
  )
  \, , 
  \label{eq de kij}
  \end{align}
  namely
  $$k_{ij} = k - \1_{\set{k>i}} - \1_{\set{k>j}}, \, i_{jk} = i - \1_{\set{i>k}} - \1_{\set{i>j}}
  \text{ and }
  j_{ik}= j - \1_{\set{j>k}} - \1_{\set{j>i}}.
  $$
  Let $\sigma_k$ be the permutation of $ \set{1,\dots,d-2}$ given by
  \begin{align*}
 \left(  \begin{array}{ccccccc}
 2 & \dots &k_{ij} &1 & k_{ij}+1 & \dots & d-2
  \\
 1 & \dots &k_{ij}-1 & k_{ij} &k_{ij}+1 &\dots &d-2
  \end{array} 
  \right)
 \end{align*}
 which is the composition of $k_{ij}-1$ transpositions.
 Applying $\sigma_k^{-1}$ to the row of $Q^{(\{j,i\},\{j,i\} )}$, we obtain a matrix denoted simply
 $Q^{(\{j,i\},\{j,i\} )}_{\sigma_k(\cdot),\cdot}$ whose first row is $Q_{k,\cdot}^{(\{j,i\},\{j,i\} )}$, and we have
 \begin{align*}
 \det Q^{(\{j,i\},\{j,i\} )} = (-1)^{k_{ij}-1} \det Q^{(\{j,i\},\{j,i\} )}_{\sigma_k(\cdot),\cdot}
 \end{align*}
  Since $\mu Q = 0$, we have $Q_{k,\cdot} = - \sum_{\ell \neq k}\frac{\mu_\ell}{\mu_k}Q_{\ell,\cdot}$ and then,
  \begin{align*}
   \det Q^{(\{j,i\},\{j,i\} )}_{\sigma_k(\cdot),\cdot}
  = - \sum_{\ell \neq k} \frac{\mu_\ell}{\mu_k}
   \left | \begin{array}{c}
   Q_{\ell,\cdot}
   \vspace{2pt}
   \\
   Q^{(\{j,i\},\{j,i\} )}_{\sigma_k(2),\cdot}
   \\
   \vdots
   \\
   Q^{(\{j,i\},\{j,i\} )}_{\sigma_k(d-2),\cdot}
   \end{array}
   \right |
   = - \frac{\mu_i}{\mu_k}
     \left | \begin{array}{c}
   Q_{i,\cdot}
   \vspace{2pt}
   \\
   Q^{(\{j,i\},\{j,i\} )}_{\sigma_k(2),\cdot}
   \\
   \vdots
   \\
   Q^{(\{j,i\},\{j,i\} )}_{\sigma_k(d-2),\cdot}
   \end{array}
   \right |
   - \frac{\mu_j}{\mu_k}
 \left | \begin{array}{c}
   Q_{j,\cdot}
   \vspace{2pt}
   \\
   Q^{(\{j,i\},\{j,i\} )}_{\sigma_k(2),\cdot}
   \\
   \vdots
   \\
   Q^{(\{j,i\},\{j,i\} )}_{\sigma_k(d-2),\cdot}
   \end{array}
   \right |.
  \end{align*}
  Let $\sigma_i$ (resp. $\sigma_j$)be constructed as $\sigma_k$ but with $i_{jk}$ (resp. $j_{ik}$) instead of $k_{ji}$ then
  one observes
  \begin{align*}
    \det Q^{(\{j,i\},\{j,i\} )}_{\sigma_k(\cdot),\cdot}
  &=  - \frac{\mu_i}{\mu_k}
  \det Q^{(\{j,k\},\{j,i\} )}_{\sigma_i(\cdot),\cdot}
   - \frac{\mu_j}{\mu_k}
    \det Q^{(\{i,k\},\{j,i\} )}_{\sigma_j(\cdot),\cdot}
    \\
    &=
     - \frac{\mu_i}{\mu_k} (-1)^{i_{jk}-1}
  \det Q^{(\{j,k\},\{j,i\} )}
   - \frac{\mu_j}{\mu_k} (-1)^{j_{ik}-1}
    \det Q^{(\{i,k\},\{j,i\} )}
  \end{align*}
  We compute that
  \begin{align*}
   (-1)^{i_{jk}-1+k_{ij}-1+i^{(j)}+k^{(j)}}=-1
  \text{ and }
   (-1)^{j_{ik}-1+k_{ij}-1+j^{(i)}+k^{(i)}}=-1,
  \end{align*}
%
  leading to
  \begin{align*}
  \mu_k \mathfrak{Q}^j_{\ls{i}{j} ,\ls{i}{j} }
  = \mu_{i}   \mathfrak{Q}^j_{\ls{i}{j} ,\ls{k}{j} }
  +
  \mu_{j}   \mathfrak{Q}^i_{\ls{j}{i} ,\ls{k}{i} }.
  \end{align*}

\eproof

\subsection{Enlargement of a filtration along a sequence of increasing stopping times} \label{mrt}
We fix a strategy $\phi \in \Phi$ and we study filtrations $\F^i, i \ge 0$ and $\F^\infty$ which are constructed in subsection \eqref{problem}.\\
For each $n \ge 0$, we define a new filtration $\G^n = (\cG^n_t)_{t \ge 0}$ by the relations $\cG^0_t = \cF^0_t$ and for $n \ge 1$, $\cG^n_t = \cF^0_t \vee \sigma(X_i, i \le n) = \cG^{n-1}_t \vee \sigma(X_n)$. 

\subsubsection{Representation Theorems}

The goal of this section is to derive Integral Representation Theorems for filtrations $\F^i, i \ge 0$ and $\F$.\\
We first recall, see \cite{AJ17}:

\begin{Theorem}[Lévy]
  Let $(\Omega,\cF,\F,\P)$ a filtered probability space with $\F$ non necessarily right-continuous. Let $\xi \in \cF$ and $X$ a $\F-$supermartingale.
  \begin{enumerate}
  \item We have $\espcond{\xi}{\cF_t} \to \espcond{\xi}{\cF_\infty}$ a.s. and in $L^1$, as $t \to \infty$.
  \item If $t_n$ decreases to $t$, we have $X_{t_n} \to X_{t^+}$ a.s. and in $L^1$ as $n \to \infty$.
  \end{enumerate}
  In particular, if $X_t = \espcond{\xi}{\cF_t}$, we get that $\espcond{\xi}{\cF_{t_n}} \to \espcond{\xi}{\cF_{t^+}}$ a.s. and in $L^1$ as $n \to \infty$, for $t_n$ decreasing to $t$.
\end{Theorem}

We now recall an important notion of coincidence of filtrations between two stopping times, introduced in \cite{AJ17}. This will be useful for our purpose in the sequel.\\
Let $S, T$ two random times, which are stopping times for two filtrations $\H^1 = (\cH^1_t)_{t \ge 0}$ and $\H^2 = (\cH^2_t)_{t \ge 0}$. We set
\begin{align*}
  \llbracket S, T \llbracket \, := \left\{ (\omega, s) \in \Omega \times \R_+ : S(\omega) \le s < T(\omega) \right\}.
\end{align*}
We say that $\H^1$ and $\H^2$ \emph{coincide} on $\llbracket S, T \llbracket$ if
\begin{enumerate}
\item for each $t \ge 0$ and each $\cH^1_t$-measurable variable $\xi$, there exists a $\cH^2_t$-measurable variable $\chi$ such that $\xi 1_{S \le t < T} = \chi 1_{S \le t < T}$,
\item for each $t \ge 0$ and each $\cH^2_t$-measurable variable $\chi$, there exists a $\cH^1_t$-measurable variable $\xi$ such that $\chi 1_{S \le t < T} = \xi 1_{S \le t < T}$.
\end{enumerate}

We now study the right-continuity of the filtration $\G^n$ for some $n \ge 0$. Using its specific structure, it is easy to compute conditional expectations. Lévy's theorem then allows to obtain the right-continuity.
\begin{Lemma}
  Let $n \ge 0$.
  \begin{enumerate} 
  \item If $\xi \in L^1(\cF^0_\infty)$ and $\xi' \in L^1(\sigma(X_i, 1 \le i \le n))$, then for $t \ge 0$, we have $\espcond{\xi\xi'}{\cG^n_t} = \espcond{\xi}{\cF^0_t}\xi'$.
  \item $\G^n$ is right-continuous.
  \end{enumerate}
\end{Lemma}
\begin{proof}
  \begin{enumerate}
  \item If $F \in \cF^0_t$ and $F' \in \sigma(X_i, 1 \le i \le n)$, we have, by independence,
    \begin{align*}
      \nonumber \esp{\xi\xi' 1_{F\cap F'}} &= \esp{\xi 1_F}\esp{\xi' 1_{F'}} \\
      \nonumber &= \esp{\espcond{\xi}{\cF^0_t}1_F}{\esp{\xi' 1_{F'}}} \\
                                           &= \esp{\xi'\espcond{\xi}{\cF^0_t} 1_{F \cap F'}}.
    \end{align*}
    Since $\{F \cap F'| F \in \cF^0_t, F' \in \sigma(X_i, 1 \le i \le n)\}$ is a $\pi$-system generating $\cG^n_t$, the result follows by a monotone class argument.
  \item Let $t \ge 0$ and $t_m$ decreasing to $t$. We have, using Lévy's Theorem, the previous point and the right-continuity of $\F^0$,
    \begin{align*}
      \nonumber \espcond{\xi\xi'}{\cG^n_{t^+}} &= \lim_m \espcond{\xi\xi'}{\cG^n_{t_m}} = \lim_m \xi'\espcond{\xi}{\cF^0_{t_m}} \\
      \nonumber &= \xi'\espcond{\xi}{\cF^0_t} = \espcond{\xi\xi'}{\cG^n_t}.
    \end{align*}
    By a monotone class argument, we have $\espcond{\xi}{\cG^n_{t^+}} = \espcond{\xi}{\cG^n_t}$ for all bounded $\cG^n_\infty$-measurable $\xi$, hence it follows the right-continuity of $\G^n$.
  \end{enumerate}
  \eproof
\end{proof}
Using the previous Lemma, we show how to compute conditional expectations in $\F$ and $\F^n$ for all $n \ge 0$, and show that these filtrations are right-continuous.
\begin{Proposition}
  \begin{enumerate}
  \item For all $m \ge n \ge 0$, $\F^{n}$, $\F^{m}$ and $\F^\infty$ coincide on $\llbracket 0, \tau_{n+1} \llbracket$.\\
    For all $n \ge 0$, $\F^{n}$ and $\G^{n}$ coincide on $\llbracket \tau_n, +\infty \llbracket$.
  \item For all $n \ge 0$ and $t \ge 0$, we have, for $\xi \in L^1(\cF^{n+1}_\infty)$:
    \begin{align} \label{dec-espcond}
      \espcond{\xi}{\cF^{n+1}_t} = \espcond{\xi}{\cF^n_t} 1_{t < \tau_{n+1}} + \espcond{\xi}{\cG^{n+1}_t} 1_{\tau_{n+1} \le t}.
    \end{align}
    Let $t \ge 0$ such that $\sum_{n=0}^{+\infty} \P(\tau_n \le t < \tau_{n+1}) = 1$. Then, for $\xi \in L^1(\cF^\infty_\infty)$,
    \begin{align*}
      \espcond{\xi}{\cF^\infty_t} = \sum_{n = 0}^{+\infty} \espcond{\xi}{\cF^n_t} 1_{\tau_n \le t < \tau_{n+1}}.
    \end{align*}
  \item For all $n \ge 0$, $\G^n$ is right-continuous.
  \item The filtration $\G$ is right-continuous on $[0,T]$.
  \end{enumerate}
\end{Proposition}
\begin{proof}
  \begin{enumerate}
  \item Let $t \ge 0$ be fixed.   If $\xi$ is $\cF^n_t$-measurable, since $\cF^n_t \subset \cF^m_t \subset \cF^\infty_t$ for $m \ge n$, taking $\chi = \xi$ gives a $\cF^m_t$-measurable (resp. $\cF^\infty_t$-measurable) random variable such that $\xi 1_{t < \tau_{n+1}} = \chi 1_{t < \tau_{n+1}}$.\\
    Conversely, if $\chi$ is a $\cF^m_t$-measurable random variable, then
    \begin{align*}
      \chi = f(\tilde \chi, X_1 1_{\tau_1 \le t}, \dots, X_m 1_{\tau_m \le t}),
    \end{align*}
    for a measurable $f$ and a $\cF^0_t$-measurable variable $\tilde \chi$. Since $X_k 1_{\tau_k \le t} = 0$ on $\{t < \tau_{n+1}\}$ when $k \ge n$, one gets:
    \begin{align*}
      \nonumber \chi 1_{t < \tau_{n+1}} &= f(\tilde \chi, X_1 1_{\tau_1 \le t}, \dots, X_n 1_{\tau_n \le t}, 0, \dots, 0) 1_{t < \tau_{n+1}} \\
                              &=: \xi 1_{t < \tau_{n+1}},
    \end{align*}
    where $\xi$ is $\cF^n_t$-measurable.\\
    Last, let $\chi$ be a $\cF^\infty_t$-measurable variable. Then $\chi = f(\tilde \chi, X_{i_1} 1_{\tau_{i_1} \le t}, \dots, X_{i_N} 1_{\tau_{i_N} \le t})$ for some random $N \ge 0$ and $1 \le i_1 \le \dots \le i_N$, and the same arguments applies.

    The proof of the second claim is straightforward as one remarks that for $t \ge 0$ and $n \ge 1$, the equality $f(\xi, X_1, \dots, X_n) 1_{\tau_n \le t} = f(\xi, X_1 1_{\tau_1 \le t}, \dots, X_n 1_{\tau_n \le t}) 1_{\tau_n \le t}$ holds, since the random times $\tau_i, i \ge 0$ are nondecreasing.
  \item Let $n \ge 0$ and $\xi \in L^1(\cF^{n+1}_\infty)$.\\
    Since $\F^n$ and $\F^{n+1}$ coincide on $\llbracket 0, \tau_{n+1} \llbracket$, we have $\espcond{\xi}{\cF^{n+1}_t} 1_{t < \tau_{n+1}} = \tilde \xi 1_{t < \tau_{n+1}}$ for a $\cF^n_t$-measurable variable $\tilde \xi$. In particular, the left hand side is also $\cF^n_t$-measurable. Hence $\espcond{\xi}{\cF^{n+1}_t} 1_{t < \tau_{n+1}} = \espcond{\espcond{\xi}{\cF^{n+1}_t} 1_{t < \tau_{n+1}}}{\cF^n_t} = \espcond{\xi}{\cF^n_t} 1_{t < \tau_{n+1}}$.\\
    Similarly, since $\F^{n+1}$ and $\G^{n+1}$ coincide on $\llbracket \tau_{n+1}, +\infty \llbracket$, we have $\espcond{\xi}{\cG^{n+1}_t}1_{\tau_{n+1} \le t} = \hat \xi 1_{\tau_{n+1} \le t}$ for a $\cF^{n+1}_t$-measurable variable $\hat \xi$. In particular, the left hand side is $\cF^{n+1}_t$-measurable. Hence $\espcond{\xi}{\cG^{n+1}_t}1_{\tau_{n+1} \le t} = \espcond{\espcond{\xi}{\cG^{n+1}_t}1_{\tau_{n+1} \le t}}{\cF^{n+1}_t} = \espcond{\xi}{\cF^{n+1}_t} 1_{\tau_{n+1} \le t}$.\\
    Let $t \ge 0$ such that $\sum_n \P(\tau_n \le t < \tau_{n+1}) = 1$. We have, since $\G$ and $\G^n$ coincide on $\llbracket 0, \tau_{n+1} \llbracket$, using the same arguments as before,
    \begin{align*}
      \nonumber \espcond{\xi}{\cG_t} &= \sum_n \espcond{\xi}{\cG_t} 1_{\tau_n \le t < \tau_{n+1}} = \sum_n \espcond{\xi}{\cG^n_t} 1_{\tau_n \le t < \tau_{n+1}}.
    \end{align*}
  \item We prove by induction that $\F^n$ is right-continuous. Since $\F^0$ is the augmented Brownian filtration, the result is true for $n = 0$.\\
    Assume now that $\F^{n-1}, n \ge 1,$ is right-continuous. Let $t \ge 0, \xi \in L^1(\cF^n_\infty)$ and $(t_m)_{m \ge 0}$ such that $t_m \ge t_{m+1}$ and $\lim_m t_m = t$. We have, using the previous point and the right-continuity of $\F^{n-1}$ and $\G^n$:
    \begin{align*}
      \nonumber \espcond{\xi}{\cF^n_{t^+}} &= \lim_m \espcond{\xi}{\cF^n_{t_m}} \\
      \nonumber &= \lim_m \espcond{\xi}{\cF^n_{t_m}} 1_{t_m < \tau_n} + \espcond{\xi}{\cF^n_{t_m}} 1_{\tau_n \le t_m} \\
      \nonumber &= \lim_m \espcond{\xi}{\cF^{n-1}_{t_m}} 1_{t_m < \tau_n} + \espcond{\xi}{\cG^n_{t_m}} 1_{\tau_n \le t_m} \\
      \nonumber &= \espcond{\xi}{\cF^{n-1}_t} 1_{t < \tau_n} + \espcond{\xi}{\cG_{t_m}} 1_{\tau_n \le t} \\
                                           &= \espcond{\xi}{\cF^n_t}.
    \end{align*}
  \item Let $t < T, \xi \in L^1(\cF^\infty_\infty), (t_m)_{m \ge 0}$ such that $T > t_m > t_{m+1}$ and $\lim_m t_m = t$. We have, by Lévy's Theorem and the first point,
    \begin{align*}
      \nonumber \espcond{\xi}{\cF^\infty_{t^+}} &= \lim_m \espcond{\xi}{\cF^\infty_{t_m}} 
= \lim_m \sum_{n=0}^{+\infty} \espcond{\xi}{\cF^\infty_{t_m}} 1_{\tau_n \le t_m < \tau_{n+1}} \\
                                                &= \lim_m \sum_{n=0}^{+\infty} \espcond{\xi}{\cF^n_{t_m}} 1_{\tau_n \le t_m < \tau_{n+1}}.
    \end{align*}
    Fix $\omega \in \Omega$. We have that $t_m < \hat T < \tau_{N+1}(\omega),$ hence
    \begin{align*}
      \nonumber \espcond{\xi}{\cF^\infty_{t^+}}\!\!(\omega) &= \lim_m \sum_{n=0}^{+\infty} \espcond{\xi}{\cF^n_{t_m}}\!\!(\omega) 1_{\tau_n(\omega) \le t_m < \tau_{n+1}(\omega)} \\
      \nonumber &= \lim_m \sum_{n=0}^{N(\omega)+1} \espcond{\xi}{\cF^n_{t_m}}\!\!(\omega) 1_{\tau_n(\omega) \le t_m < \tau_{n+1}(\omega)} \\
      \nonumber &= \sum_{n=0}^{N(\omega)+1} \lim_m \espcond{\xi}{\cF^n_{t_m}}\!\!(\omega) 1_{\tau_n(\omega) \le t_m < \tau_{n+1}(\omega)} \\
                                       &= \sum_{n=0}^{+\infty} \lim_m \espcond{\xi}{\cF^n_{t_m}}\!\!(\omega) 1_{\tau_n(\omega) \le t_m < \tau_{n+1}(\omega)}.
    \end{align*}
    Finally using the right-continuity of each $\F^n$, we get
    \begin{align*}
      \nonumber \espcond{\xi}{\cF^\infty_{t^+}} &= \sum_{n=0}^{+\infty} \lim_m \espcond{\xi}{\cF^n_{t_m}} 1_{\tau_n \le t_m < \tau_{n+1}} = \sum_{n=0}^{+\infty} \espcond{\xi}{\cF^n_t} 1_{\tau_n \le t < \tau_{n+1}} = \espcond{\xi}{\cF^\infty_t},
    \end{align*}
    which proves that $\F^\infty$ is right-continuous on $[0,T]$.
  \end{enumerate}
  \eproof
\end{proof}

\begin{Lemma}
  Let $n \ge 0$ and $\xi \in L^1(\cF^n_\infty)$. Let $\sigma$ be a $\F^n-$stopping time. We have:
  \begin{align}
    \espcond{\xi}{\cF^{n+1}_\sigma} = \espcond{\xi}{\cF^n_\sigma}.
  \end{align}
\end{Lemma}

\begin{proof}
  Assume first that $\sigma = s$ is deterministic.\\
  Let $\tilde\xi = \psi(\chi, X_{n+1} 1_{\tau_{n+1} \le s})$ be a $\cF^{n+1}_s$-measurable bounded variable, where $\chi$ is $\cF^n_s$-measurable. We need to show
  \begin{align*}
    \esp{\xi\tilde\xi} = \esp{\espcond{\xi}{\cF^n_s}\tilde\xi}.
  \end{align*}
  We have, with $\hat\psi(y) := \int \psi(y,x) \P_{X_{n+1}}(\ud x)$,
  \begin{align*}
    \nonumber \esp{\espcond{\xi}{\cG^n_s}\psi(\chi, X_{n+1} 1_{\tau_{n+1} \le s})} &= \esp{\espcond{\xi}{\cG^n_s}\psi(\chi, 0) 1_{s < \tau_{n+1}}} + \esp{\espcond{\xi}{\cG^n_s}\psi(\chi, X_{n+1}) 1_{\tau_{n+1} \le s}} \\
                                                                         &= \esp{\xi \psi(\chi, 0) 1_{s < \tau_{n+1}}} + \esp{\xi \hat\psi(\chi)1_{\tau_{n+1} \le s}},
  \end{align*}
  and the same computation with $\xi$ instead of $\espcond{\xi}{\cG^n_s}$ gives the same result.\\
  Let $\sigma$ be a $\F^n$-stopping time, and let $\xi_s = \espcond{\xi}{\cF^n_s} = \espcond{\xi}{\cF^{n+1}_s}$. Since $\F^n$ (or $\F^{n+1}$) is right-continuous, there exists a right-continuous modification of $(\xi_s)_{s \ge 0}$. Applying Doob's Theorem twice gives $\xi_\sigma = \espcond{\xi}{\cF^n_\sigma}$ and $\xi_\sigma = \espcond{\xi}{\cF^{n+1}_\sigma}$, hence we get the result.
  \eproof
\end{proof}
We are now in position to prove an Integral Representation Theorem in the filtrations $\F^n$, for all $n \ge 0$.
\begin{Proposition}
  Let $n \ge 0$ and $\xi \in L^2(\cF^n_T)$. Then there exists a $\G^n-$predictable process $\psi$ such that
  \begin{align*}
    \xi = \espcond{\xi}{\cF^n_{T \wedge \tau_n}} + \int_{T \wedge \tau_n}^T \psi_s \ud W_s.
  \end{align*}
\end{Proposition}

\begin{proof}
  We prove the theorem by induction on $n \ge 0$, following ideas from \cite{A00}. The case $n = 0$ is the usual Martingale Representation Theorem in the augmented Brownian filtration $\F^0$.\\
  Assume now that the statement is true for all $\xi \in L^2(\cF^{n-1}_T)\, (n \ge 1)$. Let $\xi \in L^2(\cF^n_T)$.\\
  Since $\cF^n_T = \cF^{n-1}_T \vee \sigma(X_n 1_{\tau_n \le T})$, we get that $\xi = \lim_{m \to \infty} \xi_m$ in $L^2(\cF^n_T)$, with $\xi_m = \sum_{i=1}^{l_m} \chi^i_m \zeta^i_m$ and $(\chi^i_m,\zeta^i_m) \in L^\infty(\cF^{n-1}_T) \times L^\infty(\sigma(X_n 1_{\tau_n \le T}))$ for all $m \ge 0$ and $1 \le i \le l_m$.\\
  By induction, there exist $\F^{n-1}$-predictable processes $\psi^{i,m}$ such that $\chi^i_m = \espcond{\chi^i_m}{\cF^{n-1}_{T \wedge \tau_{n-1}}} + \int_{T \wedge \tau_{n-1}}^T \psi^{i,m}_s \ud W_s$. Since $\tau_n$ is a $\F^{n-1}$-stopping time with $\tau_n \ge \tau_{n-1}$, we get:
  \begin{align*}
    \chi^i_m = \espcond{\chi^i_m}{\cF^{n-1}_{T \wedge \tau_n}} + \int_{T \wedge \tau_n}^T \psi^{i,m}_s \ud W_s.
  \end{align*}
  Since $\zeta^i_m \in L^\infty(\sigma(X_n 1_{\tau_n \le T})) \subset L^2(\cF^n_{T \wedge \tau_n})$, we get
  \begin{align*}
    \zeta^i_m \int_{T \wedge \tau_n}^T \psi^{i,m}_s \ud W_s = \int_{T \wedge \tau_n}^T \zeta^i_m \psi^{i,m}_s \ud W_s.
  \end{align*}
  In addition, since $\chi^i_m$ is $\cF^{n-1}_T$-measurable and $\zeta^i_m \in L^2(\cF^n_{T \wedge \tau_n})$, we get, by the previous lemma,
  \begin{align*}
    \zeta^i_m \espcond{\chi^i_m}{\cF^{n-1}_{T \wedge \tau_n}} = \zeta^i_m \espcond{\chi^i_m}{\cF^n_{T \wedge \tau_n}} = \espcond{\chi^i_m \zeta^i_m}{\cF^n_{T \wedge \tau_n}}.
  \end{align*}
  Summing over $1 \le i \le l_m$ gives:
  \begin{align*}
    \nonumber \xi_m &= \sum_{i=1}^{l_m} \chi^i_m \zeta^i_m \\
    \nonumber &= \sum_{i=1}^{l_m} \espcond{\chi^i_m \zeta^i_m}{\cF^n_{T \wedge \tau_n}} + \sum_{i=1}^{l_m} \int_{T \wedge \tau_n}^T \zeta^i_m \psi^{i,m}_s \ud W_s \\
          &= \espcond{\xi_m}{\cF^n_{T \wedge \tau_n}} + \int_{T \wedge \tau_n}^T \psi^m_s \ud W_s,
  \end{align*}
  where $\psi^m := \sum_{i=1}^{l_m} \psi^{i,m}_s \zeta^i_m$.\\
  Finally, since $\xi_m \to \xi$ in $L^2(\cF^n_T)$, we get that $\espcond{\xi_m}{\cF^n_{T \wedge \tau_n}} \to \espcond{\xi}{\cF^n_{T \wedge \tau_n}}$ in $L^2(\cF^n_T)$, hence $\int_{T \wedge \tau_n}^T \psi^m_s \ud W_s$ converges to a limit $\int_{T \wedge \tau_n}^T \psi_s \ud W_s$ for a $\F^n$-predictable process $\psi$. \eproof
\end{proof}

\begin{Theorem}
  Let $0 \le T \le +\infty$ and $\xi \in L^2(\cG^n_T)$. For all $0 \le k \le n$, there exists $\F^k$-predictable processes $\psi^k$ such that:
  \begin{align*}
    \nonumber \xi &= \esp{\xi} + \sum_{k = 0}^{n-1} \int_{T \wedge \tau_k}^{T \wedge \tau_{k+1}} \psi^k_s \ud W_s + \int_{T \wedge \tau_n}^T \psi^n_s \ud W_s\\
    \nonumber &+  \sum_{k = 0}^{n-1} \left(\espcond{\xi}{\cF^{k+1}_{T \wedge \tau_{k+1}}} - \espcond{\xi}{\cF^k_{T \wedge \tau_{k+1}}}\right)\\
        &= \esp{\xi} + \int_0^T \Psi^n_s \ud W_s + \sum_{k = 0}^{n-1} \left(\espcond{\xi}{\cF^{k+1}_{T \wedge \tau_{k+1}}} - \espcond{\xi}{\cF^k_{T \wedge \tau_{k+1}}}\right),
  \end{align*}
  with $\Psi^n_t := \sum_{k = 0}^{n-1} \psi^k_t 1_{T \wedge \tau_k < t \le T \wedge \tau_{k+1}} + \psi^n_t 1_{T \wedge \tau_n < t \le T}$.
\end{Theorem}
\begin{proof}
  This is an immediate consequence of the previous theorem.
  \eproof
\end{proof}
Last, we extend this theorem to obtain an Integral Representation Theorem in $\F^\infty$.\\
We now fix $\xi \in L^2(\cF^\infty_T)$ and consider the filtration $\A = (\cA_n)_{n \in \N}$ defined by $\cA_n := \cF^n_T$. We have $\cA_\infty = \bigvee_n \cA_n = \cF^\infty_T$. By Lévy's Theorem, we get
\begin{align}
  \label{lim1}
  \espcond{\xi}{\cF^n_T} = \espcond{\xi}{\cA_n} \to \espcond{\xi}{\cA_\infty} = \xi,\mbox{ a.s.}
\end{align}
For all $n \ge 0$, since $\cF^n_T \subset \cF_T$, we can write:
\begin{align*}
  \nonumber \espcond{\xi}{\cF^n_T} = &\esp{\xi} + \sum_{k=0}^{n-1} \int_{T \wedge \tau_k}^{T \wedge \tau_{k+1}} \psi^{n,k}_s \ud W_s + \int_{T \wedge \tau_n}^T \psi^{n,n}_s \ud W_s \\
                                               &+ \sum_{k=0}^{n-1} \left( \espcond{\xi}{\cF^{k+1}_{T \wedge \tau_{k+1}}} - \espcond{\xi}{\cF^k_{T \wedge \tau_{k+1}}} \right).
\end{align*}

\begin{Lemma}
  We have $\psi^{n,k} = \psi^{k,k}$ on $[T \wedge \tau_k, T \wedge \tau_{k+1})$, for all $n \ge k$.
\end{Lemma}

\begin{proof} It follows easily by induction, comparing $\espcond{\xi}{\cF^k_T}$ and $\espcond{\espcond{\xi}{\cF^n_T}}{\cF^k_T}$ and using Itô's isometry.\\
  \eproof
\end{proof}

For all $n \ge 0$, we define $\psi^n := \psi^{n,n}$. Thus we have, for all $n \ge 0$,
\begin{align*}
  \nonumber \espcond{\xi}{\cG^n_T} = &\esp{\xi} + \sum_{k=0}^{n-1} \int_{T \wedge \tau_k}^{T \wedge \tau_{k+1}} \psi^k_s \ud W_s + \int_{T \wedge \tau_n}^T \psi^n_s \ud W_s \\
                                               &+ \sum_{k=0}^{n-1} \left( \espcond{\xi}{\cF^{k+1}_{T \wedge \tau_{k+1}}} - \espcond{\xi}{\cF^k_{T \wedge \tau_{k+1}}} \right). 
\end{align*}
We set, for $0 \le s \le T$,
\begin{align*}
  \Psi_s &= \sum_{k = 0}^{+\infty} \psi^k_s 1_{T \wedge \tau_k \le s < T \wedge \tau_{k+1}},\\
  \Psi^n_s &= \Psi_s 1_{s \le T \wedge \tau_{n+1}} + \psi^n_s 1_{T \wedge \tau_{n+1} < s}, \mbox{ and}\\
  \Delta^k_s &:= \espcond{\xi}{\cF^{k+1}_{s \wedge \tau_{k+1}}} - \espcond{\xi}{\cF^k_{s \wedge \tau_{k+1}}},
\end{align*}
so that
\begin{align*}
  \espcond{\xi}{\cF^n_T} = \esp{\xi} + \int_0^T \Psi^n_s \ud W_s + \sum_{k=0}^{n-1} \Delta^k_T.
\end{align*}

\begin{Theorem}[Integral Representation Theorem for $\F^\infty$]
  For $\xi \in L^2(\cF^\infty_T)$, we have
  \begin{align*}
    \xi = \esp{\xi} + \int_0^T \Psi_s \ud W_s + \sum_{k = 0}^{+\infty} \Delta^k_T.
  \end{align*}
\end{Theorem}

\begin{proof}
  By definition of $N = N^\phi_T$, we have $T < \tau_{n+1}$ on $\{n \ge N\}$, see Section \ref{game}. Thus,
  \begin{align*}
    \nonumber 1_{N \le n} \int_0^T \Psi^n_s \ud W_s &= \left( \int_0^{T \wedge \tau_{n+1}} \Psi_s \ud W_s + \int_{T \wedge \tau_{n+1}}^T \psi^n_s \ud W_s\right) 1_{N \le n} 
                                                    = 1_{N \le n} \int_0^T \Psi_s \ud W_s.
  \end{align*}
  Moreover, if $k \ge n$, we have, since $T \wedge \tau_{k+1} = T$,
  \begin{align*}
    \nonumber \Delta^k_T 1_{N \le n} &= \left(\espcond{\xi}{\cF^{k+1}_{T \wedge \tau_{k+1}}} - \espcond{\xi}{\cF^k_{T \wedge\tau_{k+1}}}\right) 1_{N \le n} \\
                                     &= \left(\espcond{\xi}{\cF^{k+1}_T} - \espcond{\xi}{\cF^k_T}\right) 1_{N \le n}.
  \end{align*}
  Applying \eqref{dec-espcond} to $\chi = \espcond{\xi}{\cF^{k+1}_T}$, we get
  \begin{align*}
    \chi = \espcond{\chi}{\cF^{k+1}_T} &= \espcond{\chi}{\cF^k_T} 1_{T < \tau_{k+1}} + \espcond{\chi}{\cG^{k+1}_T} 1_{\tau_{k+1} \le T}.
  \end{align*}
  Since $T < \tau_{n+1} \le \tau_{k+1}$ on $\{N \le n\}$, we finally obtain
  \begin{align*}
    \espcond{\xi}{\cF^{k+1}_T} 1_{N \le n} = \chi 1_{N \le n} = \espcond{\chi}{\cF^k_T} 1_{N \le n} = \espcond{\xi}{\cF^k_T} 1_{N \le n},
  \end{align*}
  which gives $\Delta^k_T 1_{N \le n} = 0$.
  Thus:
  \begin{align*}
    \nonumber \espcond{\xi}{\cF^n_T}1_{N \le n} &= \left(\esp{\xi} + \int_0^T \Psi^n_s \ud W_s + \sum_{k=0}^{n-1} \Delta^k_T\right) 1_{N \le n} \\
                                                &= \left(\esp{\xi} + \int_0^T \Psi_s \ud W_s + \sum_{k=0}^{+\infty} \Delta^k_T\right) 1_{N \le n}.
  \end{align*}
  Since $1_{N \le n} \to 1$ a.s. when $n \to \infty$ as $N = N^\phi_T$ and $\phi$ is an admissible strategy, see Section \ref{game}, we get, sending $n$ to $+\infty$, recall \eqref{lim1},
  \begin{align*}
    \xi = \esp{\xi} + \int_0^T \Psi_s \ud W_s + \sum_{k=0}^{+\infty} \Delta^k_T.
  \end{align*} \eproof
\end{proof}

\begin{Remark}We have:
  \begin{align}
    \left[\int_0^\cdot \Psi_s \ud W_s, \sum_{k=0}^{+\infty} \Delta^k \right]_t &= 0, \\
    \left[\int_0^\cdot \Psi_s \ud W_s\right]_t &= \int_0^t \Psi^2_s \ud s, \\
    \left[\sum_{k=0}^{+\infty} \Delta^k\right]_t &= \sum_{\tau_{k+1} \le t} |\Delta^k_t|^2.
  \end{align}
  In particular, martingales $\int_0^\cdot \Psi_s dW_s$ and $\sum_k \Delta^k$ are orthogonal.
\end{Remark}

\subsubsection{Backward Stochastic Differential Equations} \label{sec BSDE}

We now consider Backward Stochastic Differential Equations. Let $\F$ be one of the filtrations $\F^i, i \ge 0$ or $\F^\infty$. Let $\xi$ be a $\cF_T$-measurable variable and $f : \Omega \times[0,T] \times \R^d \times \R^{d \times \kappa} \to \R^d$. We assume here that $\xi$ and $f$ are standard parameters \cite{EKPQ97}:
\begin{itemize}
\item $\xi \in L^2(\cF_T)$,
\item $f(\cdot, 0, 0) \in \H^2_d(\F )$,
\item There exists $C > 0$ such that
  \begin{align*}
    |f(t,y_1,z_1) - f(t, y_2, z_2)| \le C\left(|y_1-y_2| + |z_1+z_2|\right).
  \end{align*}
\end{itemize}
Under these hypothesis, since $\F$ is right-continuous, one can prove (\cite{EKPQ97}, Theorem 5.1):
\begin{Theorem} \label{ex-uni-bsde} 
  There exists a unique solution $(Y,Z,M) \in \mathbb{S}^2_d(\F) \times \H^2_{d \times \kappa}(\F) \times \H^2_d(\F)$ such that $M$ is a martingale with $M_0 = 0$, orthogonal to the Brownian motion, and satisfying
  \begin{align*}
    Y_t = \xi + \int_t^T f(s,Y_s,Z_s) \ud s - \int_t^T Z_s \ud W_s - \int_t^T \ud M_s.
  \end{align*}
\end{Theorem}
When $d=1$, one can easily deal with linear BSDEs in $\F$, and the specific form of its solutions allows to prove a Comparison Theorem. The proofs follow closely \cite{EKPQ97}, Theorem 2.2.
\begin{Theorem}
  Let $(b,c)$ be a bounded $(\R \times \R^\kappa)$-valued predictable process and let $a \in \H^2(\F)$. Let $\xi \in L^2(\cF_T)$ and let $(Y,Z,M) \in \mathbb{S}^2(\F) \times \H^2_{1\times \kappa}(\F) \times \H^2(\F)$ be the unique solution to
  \begin{align*}
    Y_t = \xi + \int_t^T \left( a_s Y_s + b_s Z_s + c_s \right) \ud s - \int_t^T Z_s \ud W_s - \int_t^T \ud M_s.
  \end{align*}
  Let $\Gamma \in \H^2(\F)$ the solution to
  \begin{align*}
    \Gamma_t = 1 + \int_0^t \Gamma_s a_s \ud s + \int_0^t \Gamma_s b_s \ud W_s.
  \end{align*}
  Then, for all $t \in [0,T]$, one has almost surely,
  \begin{align*}
    Y_t = \Gamma_t^{-1}\espcond{\Gamma_T \xi + \int_t^T \Gamma_s c_s \ud s}{\cF_t}.
  \end{align*}
\end{Theorem}
\begin{proof}
  We fix $t \in [0,T]$ and we apply Itô's formula to the process $Y_t \Gamma_t$:
  \begin{align*}
    \ud \left(Y_t \Gamma_t\right) = Y_{t^-} \ud \Gamma_t + \Gamma_{t^-} \ud Y_t + \ud \left[ Y, \Gamma \right]_t.
  \end{align*}
  Since $\Gamma$ is continuous, we get $\left[ Y, \Gamma \right]_t = \left< Y^c, \Gamma^c \right>_t + \sum_{s \le t} \left(\Delta Y_s\right)\left(\Delta\Gamma_s\right) = \left< Y^c, \Gamma \right>_t$, thus,
  \begin{align*}
    \ud \left(Y_t\Gamma_t\right) = \Gamma_t\left(b_tY_t + Z_t\right) \ud W_t + \Gamma_t \ud M_t - \Gamma_t c_t \ud t.
  \end{align*}
  We define a martingale by $N_t = \int_0^t \Gamma_s (b_s Y_s + Z_s) \ud W_s + \int_0^t \Gamma_s \ud M_s$, and the previous equality gives
  \begin{align*}
    Y_T \Gamma_T = Y_t \Gamma_t - \int_t^T \Gamma_s c_s \ud s + N_T - N_t.
  \end{align*}
  Taking conditional expectation with respect to $\cF_t$ on both sides gives the result.\\
  \eproof\end{proof}
\begin{Theorem} \label{comparison}
  Let $(\xi,f)$ and $(\xi',f')$ two standard parameters. Let $(Y,Z,M) \in \mathbb{S}^2(\F) \times \H^2_{1\times \kappa}(\F) \times \H^2(\F)$ (resp. $(Y',Z',M')$) the solution associated with $(\xi,f)$ (resp. $(\xi',f')$). Assume that
  \begin{itemize}
  \item $\xi \ge \xi'$ a.s.,
  \item $f(Y',Z',M') \ge f'(Y',Z',M')$ a.s.
  \end{itemize}
  Then $Y_t \ge Y'_t$ almost surely for all $t \in [0,T]$.
\end{Theorem}
\begin{proof}Since $f$ is Lipschitz, we consider the bounded processes $a,b$ and $c$ defined by:
  \begin{align}
    a_t &= \frac{f(t,Y_t,Z_t) - f(t,Y'_t,Z_t)}{(Y_t - Y'_t)}1_{Y_t \neq Y'_t}, \\
    b^i_t &= \frac{\left(f(t,Y'_t,Z_t) - f(t,Y'_t, Z'_t)\right)(Z_t-Z'_t)}{|Z_t-Z'_t|^2} 1_{Z_t \neq Z'_t} , \\
    c_t &= f(Y',Z',M') - f'(Y',Z',M'),
  \end{align}
  Setting $\delta Y_t = Y_t - Y'_t, \delta Z_t = Z_t - Z'_t$ and $\delta M_t = M_t - M'_t$, we observe that $(\delta Y, \delta Z, \delta M)$ is the solution to the following linear BSDE:
  \begin{align}
    \delta Y_t = \delta Y_T + \int_t^T \left( a_s \delta Y_s + b_s \delta Z_s + c_s \right) \ud s - \int_t^T \delta Z_s \ud W_s - \int_t^T \ud \delta M_s.
  \end{align}
  Using the previous Theorem, we get $Y_t = \Gamma_t^{-1} \espcond{\delta Y_T \Gamma_T + \int_t^T \Gamma_s c_s \ud s}{\cF_t}$. By definition, $\Gamma$ is a strictly positive process, and $\delta Y_t$ and $c$ are positive by hypothesis, hence $Y_t \ge 0$.\\
  \eproof\end{proof}

\bibliographystyle{plain}
\bibliography{bib}

\begin{thebibliography}{10}

\bibitem{AJ17}
Anna Aksamit and Monique Jeanblanc.
\newblock {\em Enlargement of filtration with finance in view}.
\newblock Springer, 2017.

\bibitem{A00}
J{\"u}rgen Amendinger.
\newblock Martingale representation theorems for initially enlarged
  filtrations.
\newblock {\em Stochastic Processes and their Applications}, 89(1):101--116,
  2000.

\bibitem{B15}
Philippe Biane.
\newblock Polynomials associated with finite markov chains.
\newblock In {\em In Memoriam Marc Yor-S{\'e}minaire de Probabilit{\'e}s
  XLVII}, pages 249--262. Springer, 2015.

\bibitem{carmona2010valuation}
Ren{\'e} Carmona and Michael Ludkovski.
\newblock Valuation of energy storage: An optimal switching approach.
\newblock {\em Quantitative finance}, 10(4):359--374, 2010.

\bibitem{CEK11}
Jean-Fran{\c{c}}ois Chassagneux, Romuald Elie, and Idris Kharroubi.
\newblock A note on existence and uniqueness for solutions of multidimensional
  reflected {BSDE}s.
\newblock {\em Electronic Communications in Probability}, 16:120--128, 2011.

\bibitem{CEK10}
Jean-Fran{\c{c}}ois Chassagneux, Romuald Elie, and Idris Kharroubi.
\newblock Discrete-time approximation of multidimensional {BSDE}s with oblique
  reflections.
\newblock {\em Ann. Appl. Probab.}, 22(3):971--1007, 2012.

\bibitem{chassagneux2018obliquely}
Jean-Fran{\c{c}}ois Chassagneux and Adrien Richou.
\newblock Obliquely reflected backward stochastic differential equations.
\newblock 2018.
\newblock <hal-01761991>.

\bibitem{CK96}
Jak{\v{s}}a Cvitanic and Ioannis Karatzas.
\newblock Backward stochastic differential equations with reflection and
  {Dynkin} games.
\newblock {\em The Annals of Probability}, 24(4):2024--2056, 1996.

\bibitem{DAFH17}
Tiziano De~Angelis, Giorgio Ferrari, and Sa{\"\i}d Hamad{\`e}ne.
\newblock A note on a new existence result for reflected {BSDE}s with
  interconnected obstacles.
\newblock {\em arXiv preprint arXiv:1710.02389}, 2017.

\bibitem{DHP09}
Boualem Djehiche, Sa{\"\i}d Hamad{\`e}ne, and Alexandre Popier.
\newblock A finite horizon optimal multiple switching problem.
\newblock {\em SIAM Journal on Control and Optimization}, 48(4):2751--2770,
  2009.

\bibitem{EKPPQ97}
Nicole El~Karoui, Christophe Kapoudjian, \'Etienne Pardoux, Shige Peng, and
  Marie-Claire Quenez.
\newblock Reflected solutions of backward {SDE}s, and related obstacle problems
  for {PDEs}.
\newblock {\em the Annals of Probability}, pages 702--737, 1997.

\bibitem{EKPQ97}
Nicole El~Karoui, Shige Peng, and Marie-Claire Quenez.
\newblock Backward stochastic differential equations in finance.
\newblock {\em Mathematical Finance}, 7(1):1--71, 1997.

\bibitem{GPP96}
Anne Gegout~Petit and \'Etienne Pardoux.
\newblock Equations diff{\'e}rentielles stochastiques r{\'e}trogrades
  r{\'e}fl{\'e}chies dans un convexe.
\newblock {\em Stochastics: An International Journal of Probability and
  Stochastic Processes}, 57(1-2):111--128, 1996.

\bibitem{HJ07}
Sa{\"\i}d Hamad{\`e}ne and Monique Jeanblanc.
\newblock On the starting and stopping problem: application in reversible
  investments.
\newblock {\em Mathematics of Operations Research}, 32(1):182--192, 2007.

\bibitem{HLP97}
Sa{\"\i}d Hamad{\`e}ne, Jean-Pierre Lepeltier, and Shige Peng.
\newblock {BSDEs} with continuous coefficients and stochastic differential
  games.
\newblock {\em Pitman Research Notes in Mathematics Series}, pages 115--128,
  1997.

\bibitem{HZ10}
Sa{\"\i}d Hamad{\`e}ne and Jianfeng Zhang.
\newblock Switching problem and related system of reflected backward {SDE}s.
\newblock {\em Stochastic Processes and their applications}, 120(4):403--426,
  2010.

\bibitem{HT10}
Ying Hu and Shanjian Tang.
\newblock Multi-dimensional {BSDE} with oblique reflection and optimal
  switching.
\newblock {\em Probab. Theory Related Fields}, 147(1-2):89--121, 2010.

\bibitem{KS81}
John~G Kemeny and James~Laurie Snell.
\newblock {\em Finite Markov Chains: With a New Appendix" Generalization of a
  Fundamental Matrix"}.
\newblock Springer, 1981.

\bibitem{Martyr-16}
Randall Martyr.
\newblock Finite-horizon optimal multiple switching with signed switching
  costs.
\newblock {\em Math. Oper. Res.}, 41(4):1432--1447, 2016.

\bibitem{norris1998markov}
James~Robert Norris.
\newblock {\em Markov chains}.
\newblock Number~2. Cambridge university press, 1998.

\end{thebibliography}
\end{document}